\documentclass[a4paper]{article}

\usepackage{amsmath}
\usepackage{amsfonts}
\usepackage{amssymb}
\usepackage{fullpage}
\usepackage{hyperref}

\newtheorem{theorem}{Theorem}

\newtheorem{definition}{Definition}

\newtheorem{lemma}{Lemma}

\newtheorem{proposition}{Proposition}
\newtheorem{remark}{Remark}

\newenvironment{proof}{\noindent {Proof:}}{$\,\hfill \Box$\smallskip}
\numberwithin{equation}{section}

\author{Goncalo Oliveira \\ Imperial College London \ and \ Duke University }

\title{Calabi-Yau Monopoles for the Stenzel Metric}

\date{2 August 2014}

\begin{document}
\maketitle

\begin{abstract}
We construct the first nontrivial examples of Calabi-Yau monopoles. Our main interest on these, comes from Donaldson and Segal's suggestion \cite{Donaldson2009} that it may be possible to define an invariant of certain noncompact Calabi-Yau manifolds from these gauge theoretical equations.\\
We focus on the Stenzel metric on the cotangent bundle of the $3$-sphere $T^* \mathbb{S}^3$ and study monopoles under a symmetry assumption. Our main result constructs the moduli of these symmetric monopoles and shows that these are parametrized by a positive real number known as the mass of the monopole. In other words, for each fixed mass we show that there is a unique monopole which is invariant in a precise sense. Moreover, we also study the large mass limit under which we give precise results on the bubbling behavior of our monopoles.\\
Towards the end an irreducible $SU(2)$ Hermitian-Yang-Mills connection on the Stenzel metric is constructed explicitly.
\end{abstract}

\section{Introduction}

A Calabi-Yau manifold $X^{2n}$ is a Ricci flat K\"ahler manifold with trivial canonical bundle. We shall further fix a K\"ahler form $\omega$, together with a holomorphic volume form $\Omega=\Omega_1 + i \Omega_2$ and refer to the Calabi-Yau manifold as the pair $(X^{2n}, \omega, \Omega)$. In this paper one must restrict to the case $n=3$.\\
Let $G$ be a compact, semisimple Lie group with Lie algebra $\mathfrak{g}$ and $P \rightarrow X$ a principal $G$-bundle over a Calabi-Yau manifold $(X^{2n}, \omega, \Omega)$. Denote by $\mathfrak{g}_P =P \times_{(Ad, G)} \mathfrak{g}$ the adjoint bundle and $\mathfrak{g}_P^{\mathbb{C}}$ its complexification. Equip the first of these with an $Ad$-invariant metric and the second one with the respective Hermitian metric.

\begin{definition}\label{def:ComplexMonopoles}
Let $A$ be a connection on $P$ and $\Phi = \Phi_1 + i \Phi_2 \in \Omega^0(X, \mathfrak{g}_P^{\mathbb{C}})$ a complex Higgs Field, with $\Phi_1, \Phi_2 \in \Omega^0(X, \mathfrak{g}_P)$. The pair $(A, \Phi)$ is called a complex monopole if
\begin{eqnarray}\label{CYeq1}
\ast \partial_A \Phi & = &  \frac{1}{2} F_A \wedge \Omega, \\ \label{CYeq2}
\Lambda F_A & = & \frac{i}{2} [\Phi , \overline{ \Phi} ],
\end{eqnarray}
where $\Lambda \beta = \ast (\beta \wedge \frac{\omega^2}{2})$ for $\beta \in \Omega^2(X, \mathbb{C})$ and $\ast$ is the $\mathbb{C}$-linear extension of the Hodge $\ast$ operator.
\end{definition}

\begin{definition}\label{def:RealMonopole}
A complex monopole $(A, \Phi)$ is called a Calabi-Yau monopole if $\Phi=\Phi_1$, i.e. $\Phi_2=0$, these satisfy
\begin{eqnarray}\label{realM}
\ast \nabla_A \Phi & = &   F_A \wedge \Omega_1, \\ \label{realM2}
\Lambda F_A & = & 0.
\end{eqnarray}
\end{definition}

If $\nabla_A \Phi =0$ and $\Phi \neq$, then $\Phi$ must be preserved by the holonomy of the connection $A$, which must then be reducible as $G$ is semisimple. Moreover, in this case the equations reduce to
$$F_{A} \wedge \Omega_1 = \Lambda F_A=0 \ , \ \nabla_A \Phi=0.$$
So $A$ is an Hermitian Yang Mills (HYM) connection and $\Phi$ a parallel Higgs field. It follows either from a maximum principle or an integration by parts argument that if $X$ is compact and $(A, \Phi)$ smooth, these are the unique solutions. So, in order to study irreducible monopoles one must let $X$ be either complete noncompact or the fields $(A, \Phi)$ have singularities. Notice that the gauge group $\mathcal{G}$ preserves all the equations, so we can define the moduli space of monopoles as
\begin{equation}\label{generalmoduli}
\mathcal{M}(X,P) = \lbrace (A, \Phi) \ \vert \ \text{solving \ref{CYeq1}, \ref{CYeq2} and $A$ irreducible} \rbrace / \mathcal{G}.
\end{equation}

Much of the interest in Calabi-Yau monopoles is due to the Donaldson and Segal's suggestion in \cite{Donaldson2009} that these may be used to define an enumerative invariant of noncompact Calabi-Yau manifolds. This might be related to other conjectural invariants obtained by "counting" special Lagrangian (sLag) submanifolds of $X$. Slags are real $3$ dimensional sumbamifolds calibrated by $\Omega_1= \Re(\Omega)$, and in particular they are volume minimizing in their homology class. Attempts on  defining such counts of sLags appear for example in Joyce's work, \cite{Joyce2002}. Indeed it is a result of McLean \cite{McLean1998} that a Slag $N$ is rigid if and only if $b_1(N)=0$, for example if $N$ is a rational homology sphere. Donaldson and Segal suggest that studying monopoles is a complementary picture, which may define an invariant closely related to a count of sLags. The general expectation is that under some asymptotic regime where the mass (i.e. the asymptotic value of $ \vert \Phi \vert$) gets very large, monopoles concentrate along some SLags whose homology class is determined by the topological type of the bundle $P$. Such a concentration phenomena is expected to be modeled on $\mathbb{R}^3$ monopoles along the transverse directions to the Slag.\\
This is motivated by the situation on $\mathbb{C}^3= \mathbb{R}^3 \times \mathbb{R}^3$ with the flat metric, where dimensional reduction gives examples by lifting $3$ dimensional monopoles on $\mathbb{R}^3$. Besides this, no examples of monopoles were known to exist and is this question of existence which is addressed in this work. There are also similar theories on noncompact $G_2$ manifolds relating solutions to monopole equations to coassociative cycles \cite{Oliveira13}. The work in this paper and the analytic properties of the monopole equations are work for the PhD thesis of the author \cite{Oliveira2014}.\\

It is an interesting question to find an example of an explicit Calabi-Yau manifold with nontrivial topology and interesting Slags, in order to study the monopole equations. The cotangent bundle to the $3$-sphere $T^* \mathbb{S}^3$ has such an explicit Calabi-Yau structure and has an interesting Slag, namely the zero section. The Ricci flat K\"ahler metric $g$ on $T^*\mathbb{S}^3$ is known as the Stenzel metric \cite{Stenzel1993}. This will be described in detail in section \ref{part1} and in this complex $3$ dimensional case first appeared in the literature in \cite{Can90}. This metric is highly symmetric and in fact there is a compact Lie group $K=Spin(4)$ acting on $(T^* \mathbb{S}^3, g)$ with cohomogeneity $1$, i.e. the action is by isometries and the principal orbits have codimension $1$.\\
A principal $G$-bundle $P$ on $T^*\mathbb{S}^3$ is said to be $K$-homogeneous if there is a lift of the $K$-action on the base to the total space of $P$, and in this case there is a notion of $K$-invariant pairs $(A, \Phi)$. Let $\rho: T^* \mathbb{S}^3 \rightarrow \mathbb{R}$ be the distance $\rho = dist(M, \cdot)$ to the zero section. In general one is interested in studying monopoles $(A,\Phi)$ whose mass
\begin{equation}\label{mass}
m(A,\Phi)= \lim_{\rho \rightarrow \infty} \vert \Phi \vert,
\end{equation}
is well defined and finite. So from now on we shall suppose that this holds for all pairs $(A, \Phi)$, and we shall say this pair is irreducible, if the connection $A$ is irreducible. Moreover, proposition $3.1.26$ in \cite{Oliveira2014} gives conditions under which complex monopoles reduce to Calabi-Yau monopoles. These conditions make sense on the more general class of asymptotically conical (AC) Calabi-Yau manifolds, to which $(T^* \mathbb{S}^3, g)$ belongs. Inspired by this result we shall restrict to study finite mass Calabi-Yau monopoles. Moreover, we shall refer the interested reader to \cite{Oliveira2014}, for the more general theory of finite mass Calabi-Yau monopoles in AC Calabi-Yau manifolds. It contains a detailed study of the boundary conditions and identifies examples of AC Calabi-Yau manifolds on which there may be interesting monopoles and on which there are interesting sLag spheres.

\begin{definition}\label{moduliS}
Let $\mathcal{G}_{inv}$ denote the $K$ invariant gauge transformations on $P$, then the moduli space of finite mass invariant monopoles on $P \rightarrow \Lambda^2_-(M)$ is defined as
\begin{equation}\label{invariantmoduli}
\mathcal{M}_{inv}(P) = \lbrace \text{$(A, \Phi) \ \vert \ K$-invariant, irreducible Calabi-Yau monopoles} \rbrace / \mathcal{G}_{inv}.
\end{equation}
\end{definition}

Some notation needs to be introduced in order to state the main theorem \ref{Thm:TheoremStenzel} below. The monopole equations used here are inspired by the monopole equations in $3$ dimensions. In the Euclidean $\mathbb{R}^3$ and for structure group $SU(2)$, there is a unique mass $1$ spherically symmetric solution known as the BPS monopole \cite{BPS} which will be denoted $(A^{BPS}, \Phi^{BPS})$. Moreover, for structure group $\mathbb{S}^1$ there are no smooth solutions, but a singular one known as the Dirac monopole. It will also be the case for the Calabi-Yau monopoles studied here that there are Abelian monopoles having singularities along the Slag zero section, explicit examples of these monopoles are constructed in section \ref{lbig} and will be called Dirac Calabi-Yau monopoles by analogy. The main result of the paper is a construction of nonabelian monopoles and the precise result is

\begin{theorem}\label{Thm:TheoremStenzel}
There is a homogeneous $SU(2)$-bundle $P$ over $T^* \mathbb{S}^3$, such that the space of invariant Calabi-Yau monopoles $\mathcal{M}_{inv}(P)$ is non empty and the following hold:
\begin{enumerate}
\item For all Calabi-Yau monopoles in $\mathcal{M}_{inv}(P)$, the Higgs field $\Phi$ is bounded and the mass gives a bijection
$$m : \mathcal{M}_{inv}(P) \rightarrow \mathbb{R}^+.$$
\item Let $R >0$, and $\lbrace (A_{\lambda}, \Phi_{\lambda}) \rbrace_{\lambda \in [\Lambda, +\infty)} \in \mathcal{M}_{inv}(P)$ be a sequence of Calabi-Yau monopoles with mass $\lambda$ converging to $+\infty$. Then there is a sequence $\eta(\lambda, R) \rightarrow 0$ as $\lambda \rightarrow + \infty$, such that the restriction to each fibre $T_x \mathbb{S}^3$ for $x \in \mathbb{S}^3$ of
$$\exp_{\eta}^* (A_{\lambda}, \eta \Phi_{\lambda})$$
converges uniformly to the BPS monopole $(A^{BPS}, \Phi^{BPS})$ in the ball of radius $R$ in $(\mathbb{R}^3,g_E)$.
\item Let $\lbrace (A_{\lambda}, \Phi_{\lambda}) \rbrace_{\lambda \in [\Lambda, +\infty)} \subset \mathcal{M}_{inv}(P)$ be the sequence above. Then, the sequence
$$\left( A_{\lambda}, \Phi_{\lambda}- \lambda \frac{\Phi_{\lambda}}{\vert \Phi_{\lambda} \vert} \right),$$
converges uniformly with all derivatives to a zero mass Dirac Calabi-Yau monopole on $T^* \mathbb{S}^3 \backslash \mathbb{S}^3$, i.e. a reducible, singular Calabi-Yau monopole.
\end{enumerate}
\end{theorem}

The proof of this theorem is the main purpose of the paper, which is organized as follows. After describing Stenzel's metric in section \ref{part1} we construct homogeneous bundles in section \ref{sec:MonEqsStenzel}, where we also study invariant connections and Higgs fields on these bundles. Using these invariant data as input, the Calabi-Yau monopole equations are then reduced to the ODE's in proposition \ref{ODE}. The solutions to these equations are studied in sections \ref{Conesolution}, \ref{lbig} and \ref{lsmall}, where these are solved first for the cone and then for the Stenzel metric.\\
The proof of theorem \ref{Thm:TheoremStenzel} requires rewriting the equations; this is done at the end of section \ref{lsmall} with the discussion after proposition \ref{lem:MainLem}. This lemma is the last one in a sequence of rearrangements of the equations, which reduce the relevant ODE's to the ones governing spherically symmetric Calabi-Yau monopoles on $\mathbb{R}^3$ equipped with a certain spherically symmetric metric. These equations have been analyzed in the Appendix to \cite{Oliveira13} and the results therein can be applied to the situation here. The final subsection \ref{sec:HYM} finishes with one other solution to the equations which takes $\Phi=0$ and gives an explicit formula for an $SU(2)$-irreducible Hermitian Yang Mills (HYM) connection, which to the author's knowledge was previously unknown and is an interesting result by itself.

\subsection{Acknowledgments}

I want to thank my PhD supervisor Simon Donaldson, for guidance, inspiration and having proposed me this problem. I would also like to thank Mark Haskins and Andrew Dancer for very helpful comments and suggestions on an earlier version of this document.
This work was part of my PhD thesis, I am very grateful to my sponsor FCT by the financially support via the doctoral grant with reference SFRH / BD / 68756 / 2010.

\section{Stenzel's Ricci Flat Metric}\label{part1}

This section begins with an informal discussion of the Conifold and its deformations. Later the Stenzel's Calabi-Yau structure \cite{Stenzel1993} will be computed explicitly and shown to be asymptotic to the Conifold one. We remark here that the uniqueness of Stenzel's Calabi-Yau structure was recently shown in \cite{Hein2012}.

\subsection{The Conifold and its Deformations}\label{Sasaki}

The ordinary double point in $\mathbb{C}^4$ gives rise to a Calabi-Yau cone $(C, \omega_C, \Omega_C)$, known in the physics literature as the Conifold \cite{Can90}. It is a Ricci flat K\"ahler cone $(C = \mathbb{R}^+ \times \Sigma, g_0 = d\rho^2 + \rho^2 g_{\Sigma} )$, whose link $(\Sigma, g_{\Sigma})$ is a regular Sasaki-Einstein manifold. Topologically $\Sigma \cong \mathbb{S}^3 \times \mathbb{S}^2$ is the total space of a $U(1)$-bundle over $D=\mathbb{P}^1 \times \mathbb{P}^1$ with the product Fubini-Study K\"ahler structure $\omega_D$. Let $\eta$ be the contact structure on $\Sigma$, so $g_{\Sigma} =\pi_D^*g_D + \eta \otimes \eta$, where $g_D$ is the product round metric. The curvature of the connection $\eta$ is $d\eta = 2 \pi_D^*\omega_D$, so in $H^{1,1}(D, \mathbb{Z}) \cong \mathbb{Z} \oplus \mathbb{Z}$ and $c_1 = \frac{1}{2 \pi} \left[ d\eta \right] = \frac{1}{\pi} \left[ \omega_D \right]$ represents the first Chern class of the associated complex line bundle. Since $\Sigma$ is simply connected and $c_1(- \frac{1}{2}K_{\mathbb{P}^1 \times \mathbb{P}^1})= (1,1)$, one concludes that $\Sigma$ is the total space of the unit circle bundle in $-\frac{1}{2}K_{\mathbb{P}^1 \times \mathbb{P}^1}$. The complex structure $J_C$ on the cone $C$ is the one given by viewing it as the ordinary double point in $\mathbb{C}^4$. It matches the one in $D$ along the transverse directions and rotates $\rho \partial_{\rho}$ to the Reeb vector field $\xi$. This makes $(C,g_C,J_C)$ a Ricci flat K\"ahler cone with a global K\"ahler potential $\rho^2$, so $\omega_C = \frac{1}{2} d (\rho^2 \eta)= \frac{i}{2} \partial \overline{\partial} \rho^2$. The smoothings,
$$ X_{\epsilon}=\left\lbrace F(z_1,z_2,z_3,z_3,z_4) = z_1^2 +  z_2^2 +  z_3^2 +  z_4^2 = \epsilon^2    \right\rbrace \subseteq \mathbb{C}^4 ,$$
for $\epsilon \in \mathbb{R^+}$, make it nonsingular at the expense of changing the complex structure. Topologically these are $T^* \mathbb{S}^3$ and one obtains a complex $1$-parameter family of complex structures on $T^*\mathbb{S}^3$. To see that $X_{\epsilon} \cong T^* \mathbb{S}^3$, restrict to each $X_{\epsilon}$ the function $r^2 = \sum_{i=1}^4 \vert z_i \vert^2$ taking values into $[\epsilon^2, + \infty)$ and introduce the coordinates $(x_i,y_i) \in \mathbb{R}^4 \times \mathbb{R}^4 \cong \mathbb{C}^4$, via $z_i = x_i + i y_i$. Then the real and imaginary parts of the quadratic equation for $X_{\epsilon}$ are respectively
\begin{eqnarray}\label{eq:Rplusminus}
\vert x \vert^2 = R_+^2= \frac{r^2 + \epsilon^2}{2}  \ \  , \ \ \vert y \vert^2= R_-^2=  \frac{r^2 - \epsilon^2}{2} \ \ , \ \ x \cdot y = 0.
\end{eqnarray}
This shows that the map that to $(x,y) \in \mathbb{R}^4 \times \mathbb{R}^4$ associates $\left( \frac{x}{R_+},y \right) \in \mathbb{S}^3 \times \mathbb{R}^4 \subset \mathbb{R}^4 \times \mathbb{R}^4$, restricts to $X_{\epsilon} \subset \mathbb{C}^4$ as a diffeomorphism onto $T \mathbb{S}^3 \subset \mathbb{R}^4 \times \mathbb{R}^4$.
Moreover, the level sets of $r$ are either $\Sigma = \mathbb{S}^3 \times \mathbb{S}^2$ for $r \neq \epsilon$, or the zero section $\mathbb{S}^3$ for $r=\epsilon$.\\
Regarding symmetries, $SO(4)$ acts on $\mathbb{C}^4$ by matrix multiplication preserving $F$ and $r$ and so acts on $X_{\epsilon}$. The action is transitive on each level set of $r$. In fact Stenzel's Calabi-Yau structure, is invariant under this $SO(4)$ action. This symmetry allows for the reduction of the  Monge-Amp\`ere equation to an ODE. For the purpose of constructing the metric it is irrelevant whether one considers an $SO(4)$-action or its lift to a $Spin(4)$-action. However, regarding the existence of interesting invariant connections it is convenient to work with the $Spin(4)$-action instead.

\subsection{Stenzel's Ricci Flat Metric}\label{METRIC}

Identify the Lie algebra $\mathfrak{so}(4)$ with the skewsymmetric matrices. Then, let $X_1=C_{12}, X_2=C_{13}, X_3=C_{14}, X_4=C_{23}, X_5=C_{24}, X_6=C_{34}$, where $C_{ij}$ denotes the matrix whose $(i,j)$ and $(j,i)$ entries are respectively $1,-1$ and all other vanish. These satisfy the relations $\left[ C_{ij}, C_{ik} \right] = -C_{jk}$ and $\left[ C_{ij}, C_{kl} \right] = 0$ if $i,j,k,l$ are all distinct. Let $p=(R_{+}, i R_{-} ,0,0) \in X_{\epsilon} \subset \mathbb{C}^4$, with $R_+, R_-$ defined as in equation \ref{eq:Rplusminus}, then at $p$ the isotropy subgroup is generated by exponentiating $X_6$ and this is
\begin{equation}\label{isotropy}
H_p= \left\lbrace \begin{pmatrix}
I & 0  \\
0 & A 
\end{pmatrix} \ \vert \ A \in SO(2) \right\rbrace \subseteq SO(4).
\end{equation}
One fixes a lift of $SO(4)$ to $Spin(4)$, such that the isotropy subgroup $H_p \subset SO(4)$ lifts to $H \cong U(1)$ in $Spin(4)=SU(2) \times SU(2)$, with
\begin{equation}\label{isotropy2}
H \cong \left\lbrace \gamma(t) =  \left( \begin{pmatrix}
e^{it} & 0  \\
0 & e^{-i t} 
\end{pmatrix},
\begin{pmatrix}
e^{it} & 0  \\
0 & e^{-i t} 
\end{pmatrix} \right)  \ \vert \  t \in \mathbb{R} \right\rbrace \cong U(1) .
\end{equation}
and $\frac{d \gamma }{dt} \Big\vert_{t=0}  = -2 X_6$. Using the basis for $\mathfrak{spin}(4)=\mathfrak{so}(4)$ given by the $\lbrace X_i \rbrace_{i=1}^{6}$ and its dual basis $\lbrace \theta_i \rbrace_{i=1}^{6}$, the Maurer Cartan form on $Spin(4)$ is $\theta = \sum_{i=1}^{6} \theta_i X_i$ and the $1$-form
\begin{equation}\label{eq:CanInvConn}
-\frac{i}{2} \theta^6 \in \Omega^1(Spin(4), i \mathbb{R})
\end{equation}
equips the bundle $Spin(4) \rightarrow \Sigma= Spin(4)/U(1)$ with a connection. This is the canonical invariant connection in the language of \cite{Kob69}. The tangent space to the $Spin(4)$-orbits can be identified with an $Ad$ invariant complement to the isotropy algebra $\mathfrak{h}= \langle X_6 \rangle$. Fix the one given by defining $\mathfrak{m}$ to be the span of $\lbrace X_{i} \rbrace_{i=1}^5$, then
$$\mathfrak{spin}(4)= \mathfrak{h} \oplus \mathfrak{m},$$
and extending $\mathfrak{m}$ as a left invariant distribution in $Spin(4)$ gives another point of view on the canonical invariant connection. Moreover, one can further decompose $\mathfrak{m}$ into irreducible representations of $H=U(1)$ as
\begin{equation}\label{eq:AdInvComplement}
\mathfrak{m}= \langle X_1 \rangle \oplus \langle X_2 , X_3 \rangle \oplus \langle X_4 , X_5 \rangle,
\end{equation}
where $\langle X_1 \rangle$ is the trivial representation and $\langle X_2 , X_3 \rangle \cong \langle X_4 , X_5 \rangle \cong \mathbb{C}$ with the standard weight one representation. One can check that at $p$, $\langle X_4 , X_5 \rangle$ is the tangent space to the fibres of the sphere bundle inside $T^* \mathbb{S}^3 \rightarrow \mathbb{S}^3$ (using the round metric on $\mathbb{S}^3$), while $\langle X_1 \rangle \oplus \langle X_2 , X_3 \rangle$ projects surjectively onto the tangent space to the base $\mathbb{S}^3$.

\begin{proposition}
There is a $Spin(4)$-invariant Ricci flat K\"ahler metric on $T^* \mathbb{S}^3$ with K\"ahler form
\begin{equation}\label{kahlerform}
\omega = \dot{\mathcal{G}} dr \wedge \theta^1  + \mathcal{G} ( \theta^{24} + \theta^{35} ),
\end{equation}
where $\mathcal{G}=\sqrt{r^4 - \epsilon^4} \frac{\mathcal{F} '}{2}$, $\dot{\mathcal{G}} = \frac{d \mathcal{G}}{dr}$ and $\mathcal{F}(r^2)$ is the (global) K\"ahler potential, which satisfies
\begin{equation}\label{potential}
\mathcal{F}' (r^2(t))= \frac{1}{\sinh(t) } \left( \frac{3}{4 \epsilon^2}\right)^{\frac{1}{3}}  \left( \sinh(2t) -  2t \right)^{\frac{1}{3}} ,
\end{equation}
where $t \in [0, + \infty]$ is the coordinate implicitly determined by $r^2 = \epsilon^2 \cosh(t)$.
\end{proposition}
\begin{proof}
Since $b_2(T^*\mathbb{S}^3)=0$ any K\"ahler metric has a global K\"ahler potential $\mathcal{F}(r^2)$. The proof splits into $3$ steps:\\

1) Find a ($SO(4)$-invariant) formula for the K\"ahler form in terms of $\mathcal{F}(r^2)$. To do this expand the formula for the K\"ahler form $\frac{i}{2} \partial \overline{\partial} \mathcal{F} (r^2)$ in terms of the K\"ahler potential
\begin{eqnarray}\label{eq:IntermediateKahlerForm}
\omega_C = \frac{i}{2} \mathcal{F} ' \partial \overline{\partial} (r^2) + \frac{i}{2} \mathcal{F} '' \partial (r^2) \wedge \overline{\partial} (r^2).
\end{eqnarray}
The first term is $\partial \overline{\partial} (r^2) = \sum_{i} dz^i \wedge d \overline{z}^i$ and for the second
\begin{eqnarray}\nonumber
\partial r^2 \overline{\partial} r^2 & = &  (d-\overline{\partial}) r^2 \wedge (d-\partial) r^2 =-2rdr \wedge \partial r^2 -2r \overline{\partial} r^2 \wedge dr - \partial r^2 \overline{\partial} r^2 \\ \nonumber
& = & 2r dr \wedge (\overline{\partial} - \partial) r^2 - \partial r^2 \overline{\partial} r^2. 
\end{eqnarray}
Pass the last term to the left hand side and get $\partial r^2 \overline{\partial} r^2 = r dr \wedge (\overline{\partial} - \partial) r^2$, substituting this back in equation \ref{eq:IntermediateKahlerForm} so that $\omega_C = \frac{i}{2} \mathcal{F} ' \partial \overline{\partial} r^2 + i \mathcal{F} ''  r dr \wedge (\overline{\partial} - \partial) r^2$. At $p=(R_{+}, iR_{-} ,0,0) \in X_{\epsilon} \subset \mathbb{C}^4$ one may write
\begin{eqnarray}\nonumber
dz^1 = \frac{r}{2 R_{+}} dr + iR_{-} \theta^1 \ \ & & \ \ dz^2 = - R_{+} \theta^1 + \frac{ir}{2 R_{-}} dr, \\ \nonumber
dz^3 = - R_{+} \theta^2 - i R_{-} \theta^4 \ \ & & \ \ dz^4 = - R_{+} \theta^3 - i R_{-} \theta^5 . 
\end{eqnarray}
and notice that the forms on the right hand side extend to $SO(4)$-invariant forms outside the zero section. With these relations one computes $(\overline{\partial} - \partial) r^2 = \sum_i z^i d \overline{z}^i - \overline{z}^i d z^i =  2i (R_{-}dx_2- R_{+}dy_1 ) =-4i R_{-} R_{+} \theta^1$. The same can be done for the terms $dz^i \wedge d \overline{z}^i$ and one discovers that
\begin{equation}\nonumber
\omega_C = \dfrac{r}{\sqrt{r^4 - \epsilon^4}} \left(  r^2 \mathcal{F} ' + (r^4 - \epsilon^4) \mathcal{F} '' \right) dr \wedge \theta^1  + \sqrt{r^4 - \epsilon^4} \dfrac{\mathcal{F} '}{2} ( \theta^2 \wedge \theta^4 + \theta^3 \wedge \theta^5 ),
\end{equation}
which in terms of $\mathcal{G}$ is the K\"ahler form in the statement, for a (yet) unknown $\mathcal{F}(r^2)$.\\

2) Find a formula for the holomorphic volume form. This is done on the chart $\lbrace z^i \frac{\partial F}{\partial z^i} \neq 0 \rbrace$, where recall $F= \sum_i z_i^2 $. There, it is given by $\Omega = \left( \frac{\partial F }{ \partial z^i} \right)^{-1} dz^1 \wedge ... \hat{dz^i} ... \wedge dz^4$ and one can compute it at $p$, since $z_1 \neq 0$ there. Writing the result in terms of the $SO(4)$ invariant forms
\begin{eqnarray}\label{hol}
\Re(\Omega) & = & - \left( R_+^2 \theta^{123} - R_-^2 \theta^{145} \right) - \frac{r}{2} dr \wedge \left( \theta^{25} - \theta^{34} \right) , \\ \nonumber
\Im(\Omega) & = & \frac{r}{2} \left( \frac{R_{+}}{R_{-}} dr \wedge \theta^{23} - \frac{R_{-}}{R_{+}} dr \wedge \theta^{45} \right) + R_{+} R_{-} ( \theta^{134} - \theta^{125} ).
\end{eqnarray}

3) Use the formulas computed in the previous steps to reduce the Monge-Amp\`ere equation to an ODE and solve it. This is done by combining $\frac{\omega^3}{3!}= - \frac{i}{8} \Omega \wedge \overline{\Omega}$ with the formulas for $\omega$ and $\Omega$ obtained in the first two steps. Since $\frac{i}{8} \Omega \wedge \overline{\Omega} = - \frac{ r R_+ R_- }{2} dr \wedge \theta^{12345} $ and $ \frac{\omega^3}{3!} = -\dot{\mathcal{G}} \mathcal{G}^2 dr \wedge \theta^{12345}$ the ODE is $2 \dot{\mathcal{G}} \mathcal{G}^2 = r R_+ R_-$, or in terms of the K\"ahler potential $\mathcal{F}$
\begin{equation}\label{KahlerODE}
r^2 (\mathcal{F}')^3 + \frac{ r^4-\epsilon^4}{3} \frac{d}{dr^2}(\mathcal{F}')^3 =1.
\end{equation}
Change variables to $t$ such that $r^2 = \epsilon^2 \cosh(t)$, then $\epsilon^4 \sinh^2(t) = r^4-\epsilon^4$ and $\frac{d}{dr^2} = \frac{1}{\epsilon^2 \sinh(t)} \frac{d}{dt}$. Substituting this into \ref{KahlerODE}, the ODE turns out to be
\begin{equation}\label{MAODE}
\epsilon^2 \cosh(t) (\mathcal{F}')^3 + \frac{\epsilon^2 \sinh(t)}{3} \frac{d}{dt}(\mathcal{F}')^3 =1,
\end{equation}
which can be solved by introducing an integrating factor, giving the formula in the statement for the solution.
\end{proof}

\begin{remark}\label{normalization}
In some computations to be carried out further ahead it will be useful to recall the ODE \ref{KahlerODE} in the form $2 \dot{\mathcal{G}} \mathcal{G}^2 = r R_+ R_-$.
\end{remark}

For completeness, the complex structure can also be worked out explicitly in terms of the invariant forms. This can be read out of the formulas relating the $dz_i's$ with the $\theta^i$'s and this gives $I \theta^1 = \frac{r}{2 R_- R_{+}} dr$, $Idr = - \frac{2R_+ R_-}{r} \theta^1$, $I \theta^2 = - \frac{R_-}{R_{+}} \theta^4 $, $I \theta^4 = \frac{R_+}{R_-} \theta^2$, $I \theta^3 = - \frac{R_-}{R_+} \theta^5$ and $I \theta^5 = \frac{R_+}{R_-} \theta^3$. These, together with the equation \ref{kahlerform} for the K\"ahler form, give the following expression for the metric
\begin{equation}\label{RFmetric}
g= \dot{\mathcal{G}} \frac{r}{2 R_- R_{+}} dr^2 + \dot{\mathcal{G}} \frac{2R_+ R_-}{r} \theta_1^2 + \mathcal{G} \frac{R_+}{R_-} \left( \theta_2^2 + \theta_3^2 \right) + \mathcal{G} \frac{R_-}{R_+} \left( \theta_4^2 + \theta_5^2 \right).
\end{equation}

\begin{definition}
For each $\epsilon$ define the radial function given by
\begin{equation}\label{s}
\rho(r) = \int_{\epsilon}^{r} \frac{l}{2 \mathcal{G}} dl = \int_{\epsilon}^{r} \frac{l}{\sqrt{l^4 - \epsilon^4}} \frac{1}{\mathcal{F}' ( l^2)} dl .
\end{equation}
\end{definition}

The function $\rho$ just defined is the length through a geodesic orthogonal to the principal orbits and for $\epsilon=0$ it agrees with the geodesic distance to the apex of the cone. Next one defines a function which captures the volume growth of the level sets of $\rho$. The volume form for the induced metric is given by $\mathcal{G} \frac{R_-}{R_+} \mathcal{G} \frac{R_+}{R_-} \sqrt{  \dot{\mathcal{G}} \frac{2R_+ R_-}{r} } dr \wedge \theta^{1...5}= (R_+ R_-)^2 \mathcal{F}' dr \wedge \theta^{1...5}$.

\begin{definition}
Define the radial function $h^2(\rho) = \frac{1}{\epsilon^2}  (R_+ R_-)^2 \mathcal{F}'$.
\end{definition}

\begin{remark}
For the Conifold, which corresponds to $\epsilon = 0$ one already knows the K\"ahler potential is $\rho^2$. Moreover, in this case the $SO(4)$ invariant Monge-Amp\`ere equation \ref{KahlerODE} is
\begin{equation}
r^2 (\mathcal{F}')^3 + \frac{r^4}{3} \frac{d}{dr^2}(\mathcal{F}')^3 =1.
\end{equation}
The K\"ahler potential $\mathcal{F}$ is given by $\mathcal{F} = \left( \frac{3}{2} \right)^{\frac{4}{3}} r^{\frac{4}{3}}$ and so one concludes that the geodesic distance to the apex of the cone is $\rho = \left( \frac{3}{2} \right)^{\frac{2}{3}} r^{\frac{2}{3} }$. This can be used to rewrite the Ricci Flat K\"ahler metric \ref{RFmetric} on the conifold $C$ as 
\begin{equation}\label{Conifoldmetric}
g=  d\rho^2 + \rho^2 \left(  \left( \frac{2}{3}\theta_1 \right)^2 + \left( \frac{ \theta_2}{\sqrt{3}}\right)^2 + \left( \frac{ \theta_3}{\sqrt{3}}\right)^2 + \left( \frac{ \theta_4}{\sqrt{3}}\right)^2  + \left( \frac{ \theta_5}{\sqrt{3}}\right)^2  \right) .
\end{equation}
\end{remark}

\section{The Calabi-Yau Monopole Equations}\label{sec:MonEqsStenzel}

Recall that $X_{\epsilon} \backslash r^{-1}(\epsilon) \cong \left( \epsilon ; \infty \right) \times \Sigma$, where $\Sigma= Spin(4)/ U(1)$ is homogeneous and $r$ is the coordinate on the $\left( \epsilon ; \infty \right)$ component. This section describes homogeneous bundles having invariant connections and invariant Higgs Fields. Then, these are used to compute the Calabi-Yau monopole equations and reduce them to ODE's. Background material on homogeneous bundles and invariant connections can be found for example in section $2$ of chapter X in \cite{Kob69}.

\subsection{Homogeneous $SU(2)$ Bundle}\label{liftaction}

Recall that given a Lie group $G$, a principal $G$ bundle $P$ over $\Sigma=Spin(4)/ U(1)$ is said to be $Spin(4)$-homogeneous (or just homogeneous) if there is a lift of the $Spin(4)$ action on $\Sigma$ to its total space, which commutes with the right $G$ action on $P$. In particular, $Spin(4) \rightarrow \Sigma$ is itself a homogeneous $U(1)$-bundle. In general homogeneous $SU(2)$ principal bundles over $\Sigma$ are determined by their isotropy homomorphisms $\lambda_l: U(1) \rightarrow SU(2)$ and are constructed via
\begin{equation}\label{bundleP}
P_{\lambda_l} = Spin(4) \times_{(U(1), \lambda_l)} SU(2),
\end{equation} 
where the possible group homomorphisms $\lambda_l$ are parametrized by $l \in \mathbb{Z}$ and given by 
$$\lambda_l(\theta)= \begin{pmatrix} e^{i l\theta} & 0 \\ 0 & e^{-il\theta} \end{pmatrix}.$$
By construction the $P_{\lambda_l}$ are reducible to $Spin(4)$ and each connection on the latter extends to a reducible connection on $P_{\lambda_l}$ (see \cite{Kob69}). The goal is to find invariant connections on $P_l$ which are not reducible to connections on $Spin(4)$ and it will be seen in proposition \ref{Invariantconnections}, that this is not possible for all but one $l$, which is $l=1$.

\begin{remark}\label{rem:DefL}
Let $E_l= P_{\lambda_l} \times_{(SU(2),c)} \mathbb{C}^2$, or equivalently $P_{\lambda_1} \times_{(SU(2),c^{\otimes l})} \mathbb{C}^2$, where $c$ denotes the standard representation of $SU(2)$ on $\mathbb{C}^2$. As the $P_l$'s are reducible, 
$$E_l = Spin(4) \times_{c \circ \lambda_l} \mathbb{C}^2 = L^l \oplus L^{-l},$$
splits as a sum of complex line bundles $L^l$ associated with $Spin(4)$ from the degree $l$ representation of $U(1)$ on $\mathbb{C}$. As $\Sigma$ is topologically $\mathbb{S}^2 \times \mathbb{S}^3$, the bundles $E_l$ are trivial and so do extend over $T^*\mathbb{S}^3$, i.e. when the zero section is glued back in. However, the splitting above only holds outside the zero section in $T^* \mathbb{S}^3$, as the bundle $L$ itself does not extend.
\end{remark}

Recall the canonical invariant connection $ -\frac{i}{2} \theta^6 \in \Omega^1(Spin(4), i \mathbb{R})$ on $Spin(4) \rightarrow \Sigma$ defined in equation \ref{eq:CanInvConn}. This is a $U(1)$ connection and the next step is to extend it to a reducible connection on each $P_{\lambda_l}$.

\begin{definition}\label{defCanonical}
Let $T_1,T_2,T_3$ be a basis for $\mathfrak{su}(2)$ such that $[T_i,T_j] = 2 \epsilon_{ijk}T_k$. Then, the canonical invariant connection on $P_{\lambda_l}$ is
\begin{equation}\label{Canonical}
A_c^l = -\frac{l\theta^6}{2} \otimes T_1 \in \Omega^1(Spin(4), \mathfrak{su}(2)).
\end{equation}
\end{definition}

\begin{lemma}
The curvature of the canonical invariant connection $A_c^l$ is
\begin{equation}\label{Canonicalcurvature}
F_{c}^l = -\frac{l}{2} \left( \theta^{23} + \theta^{45} \right) \otimes T_1.
\end{equation}
\end{lemma}
\begin{proof}
This follows from the Maurer-Cartan relation $d\theta^6 = \theta^{23} + \theta^{45}$, the other ones are $d \theta^1 = \theta^{24} + \theta^{35}$, $d\theta^3 = - \theta^{15} - \theta^{26}$, $ d\theta^5 = \theta^{13} - \theta^{46}$, $d\theta^2 = -\theta^{14}+ \theta^{36}$.
\end{proof}
In the same way one computes $c_1(L) = \frac{1}{4 \pi} \left[ \theta^{23} + \theta^{45} \right]$, and this can be compared this with the transverse K\"ahler structure. The vector field $X_1$ is the infinitesimal generator of a free $\mathbb{S}^1$-action on $\Sigma$ and this is precisely the flow of the Reeb field. The contact form equips the bundle $\Sigma \rightarrow D$ with a connection which needs to be proportional to $ \theta^1$, and one can read from \ref{Conifoldmetric} that $\omega_D = \frac{1}{3} \left( \theta^{24} + \theta^{35} \right)$. Moreover, since $\omega_D = \frac{d \eta}{2}$, one discovers from the Maurer Cartan relations that $\eta=- \frac{2}{3} \theta^1$, as expected from \ref{Conifoldmetric} and so $c_1(\Sigma)= 2 c_1(D) = \frac{1}{3 \pi} \left[  \theta^{24} + \theta^{35}  \right]$.

\begin{remark}
In fact $L$ is the pull back of a holomorphic line bundle $\mathcal{L}$ over $D$. Moreover, $- i \frac{\theta^6}{2}$ is then a Hermitian Yang Mills connection on $\mathcal{L} \rightarrow D$ and in the case of the Conifold $C$ it does lift to a reducible Calabi-Yau monopole. In fact one wants to construct Calabi-Yau monopoles whose connection $A$ is asymptotic to $A_{\infty}= A_c^l$. This is a familiar situation in the general setting of AC Calabi-Yau monopoles described in \cite{Oliveira2014}. The class $c_1(L) \in H^2(X, \mathbb{Z})$ is defined there to be a monopole class and this is the Calabi-Yau analog of the $3$ dimensional monopole charge.
\end{remark}

\subsection{Invariant Connections and Higgs Fields}\label{section:Invariant}

The problem of finding invariant connections on $P_l$ is an application of Wang's theorem, for which the reader is referred to \cite{Kob69}.

\begin{proposition}\label{Invariantconnections}
Let $A^l \in \Omega^1(Spin(4), \mathfrak{su}(2))$ be the connection $1$ form of an invariant connection on $P_l$. Then it is left-invariant and can be written as
\begin{equation}
A^l= A_c^l + (A- A_c)
\end{equation}
where $(A- A_c) \in \mathfrak{m}^* \otimes \mathfrak{su}(2)$, extended as a left-invariant $1$-form with values in $\mathfrak{su}(2)$ is given by $A - A_c = A_1 \theta^1 \otimes T_1$ if $l \neq 1$, while if $l=1$
\begin{eqnarray}\nonumber
A - A_c & = & A_1 \theta^1 \otimes T_1 \\ \nonumber
& &+ \left( A_2 \theta^2 - A_3 \theta^3 + A_4 \theta^4 - A_5 \theta^5 \right) \otimes T_2   \\ \nonumber
& & + \left( A_3 \theta^2 + A_2 \theta^3 + A_5 \theta^4 + A_4 \theta^5 \right) \otimes T_3,
\end{eqnarray}
and $A_1,A_2,A_3,A_4,A_5 \in \mathbb{R}$.
\end{proposition}
\begin{proof}
By Wang's theorem \cite{Kob69}, invariant connections are given by morphisms of $U(1)$ representations
$$\Lambda_l : (\mathfrak{m}, Ad) \longrightarrow (\mathfrak{su}(2), Ad \circ \lambda_l).$$
Then by extending $\Lambda_l$ as a left invariant $\mathfrak{su}(2)$-valued $1$-form in $Spin(4)$ one obtains an invariant connection $A= A^l_c + \Lambda_l$ on $P_l$ (notice that $\Lambda_l=0$ gives the canonical invariant connection). Let $c$ be the standard, weight $1$, $U(1)$ representation on $\mathbb{C} \cong \mathbb{R}^2$. Split the representations above into irreducibles $\mathfrak{m} \cong \mathbb{R} \oplus c \oplus c$, and $\mathfrak{su}(2) \cong \mathbb{R} \oplus c^{\otimes l}$, where in the first of these $c \oplus c \cong \langle X_2 , X_3 \rangle \oplus \langle X_4 , X_5 \rangle$, from equation \ref{eq:AdInvComplement}. Then, Schur's lemma states that $\Lambda$ should restrict to each piece as an isomorphism or as $0$. So for $l \neq 1$, $\lambda_l= A_1 T_1 \oplus 0$, while for $l=1$, $\Lambda_1 = A_1 T_1 \oplus 1_1 \oplus 1_2$, where $A_1 \in \mathbb{R}$ and $1_1$ and $1_2$ are isomorphisms matching the $c$ components in both sides. Using the basis of $\mathfrak{m}$ given by the $X_i$'s as in section \ref{METRIC} and the basis for $\mathfrak{su}(2)$ given by the $T_i$'s as in definition \ref{defCanonical}, $1_1 , 1_2$ can be written
\begin{eqnarray}\nonumber
1_1 & = & \left( A_2 \theta^2 - A_3 \theta^3  \right) \otimes T_2 +  \left( A_3 \theta^2 + A_2 \theta^3 \right) \otimes T_3 \\ \nonumber
1_2 & = & \left(  A_4 \theta^4 - A_5 \theta^5 \right) \otimes T_2 + \left( A_5 \theta^4 + A_4 \theta^5 \right) \otimes T_3 ,
\end{eqnarray}
with $A_1,A_2,A_3,A_4, A_5 \in \mathbb{R}$. Rearranging gives the result in the statement.
\end{proof}

\begin{proposition}\label{invariantHiggs}
For all $l \in \mathbb{Z}$, there are invariant Higgs fields $\Phi$ and these are of the form $\Phi = \phi \ T_1$, with $\phi \in \mathbb{R}$.
\end{proposition}
\begin{proof}
The adjoint bundle is constructed via $\mathfrak{g}_{P_l}  \times_{(SU(2), Ad)} \mathfrak{su}(2)$ and unwinding the construction of $P$ in equation \ref{bundleP}, gives
$$ \mathfrak{g}_{P_l} = Spin(4) \times_{U(1), Ad \circ \lambda_l} \mathfrak{su}(2).$$
So, think of Higgs fields (sections of $\mathfrak{g}_{P_l}$) as functions in $Spin(4)$ with values in $\mathfrak{su}(2)$ which are equivariant for the $U(1)$ right-action on $Spin(4)$ and $Ad \circ \lambda_l$-action on $\mathfrak{su}(2)$ via $Ad \circ \lambda_l$. For $Spin(4)$-invariant Higgs fields, these functions must be constant. So the previous equivariance condition reduces to the statement that such a constant must be fixed by the $Ad \circ \lambda_l$-action, i.e. it must lie in a irreducible component given by the trivial representation. There is only one such and is the direction singled out by $T_1$.
\end{proof}

Then a $Spin(4)$-invariant pair $(A, \Phi)$ on the pull back of $P_l$ to $(\epsilon, + \infty ) \times \Sigma$ can be written as
$$A= dr \otimes A_r (r) + A_{\Sigma} (r) \ \ , \ \ \Phi = \phi(r) \otimes T_1,$$
with $A_{\Sigma}$ a $1$-parameter family as in proposition \ref{Invariantconnections} and $A_r , \Phi$ are $1$-parameter families as in proposition \ref{invariantHiggs}, parametrized by $r \in (\epsilon , \infty)$. Moreover, one can always get rid of the radial component in $A$ via a gauge transformation $g$ that only depends on the $r$-direction. To achieve this we need to solve $(g \cdot A) (\partial_r) = 0$, which can be written as $g^{-1} \frac{\partial g }{\partial r} + g^{-1} A_r g =0$, and so amounts to solving an ODE for $g$. This can always be solved with the condition $\lim_{r \rightarrow \infty}g(r)=1_{SU(2)}$, the solution is unique and so there is no loss in assuming that $A_r=0$.

\begin{remark}\label{RemarkGauge}
For the proof of theorem \ref{Thm:TheoremStenzel} one must consider invariant gauge transformations. The gauge-fixing above uses an invariant gauge transformation such that $\lim_{r \rightarrow \infty}g(r)=1_{SU(2)}$, which is a usual requirement in monopole problems, but not here. So one can still use a gauge transformation $g'$ which must not depend on $r$ and be invariant, i.e. $g$ must be a constant is the subgroup $Z_{U(1)}(SU(2)) = U(1) \subset SU(2)$ of those elements which are centralized by $U(1)$. These do not affect the radial gauge fixing above, they preserve $A_c^l$ and act by conjugation as $g (A^l - A_c^l)g^{-1}$ and so one can get rid of one the $A_i$'s. The choice of such a gauge will be postponed to a later stage, where a particular choice will ease the computations.
\end{remark}

\begin{lemma}\label{lema1}
For $l \neq 1$, the curvature of an invariant connection $A$ on $P_l$ is given by
\begin{equation}\label{Curvaturel}
F^l = \left( - \frac{l}{2}  (\theta^{23} + \theta^{45}) + \dot{A}_1 dr \wedge \theta^1 + A_{1} (\theta^{24} + \theta^{35}) \right) \otimes T_1,
\end{equation}
in particular the connection is always reducible for $l \neq 1$. For $l=1$, the curvature is
\begin{eqnarray}\nonumber
F_A & = & \left( \left(2(A_2^2 + A_3^2)  - \frac{1}{2} \right)\theta^{23} + \left( 2 (A_4^2 + A_5^2 ) - \frac{1}{2} \right) \theta^{45} \right) \otimes T_1 \\ \nonumber
& &+ \left( 2(A_2 A_4 + A_5 A_3)(\theta^{25}- \theta^{34} ) + \left( A_1 +  2(A_2 A_5-A_4 A_3) \right) (\theta^{24} + \theta^{35} ) \right) \otimes T_1 \\ \nonumber
& & +  (A_4 - 2A_1 A_3) (T_2\otimes \theta^{12} + T_3 \otimes \theta^{13}   ) + (A_5 + 2A_1A_2)( T_3 \otimes \theta^{12} - T_2 \otimes \theta^{13} ) \\ \nonumber  
& & -(A_2 + 2A_1 A_5)( T_2  \otimes \theta^{14} + T_3 \otimes \theta^{15} ) - (A_3 - 2A_1A_4)(T_3  \otimes \theta^{14}-T_2  \otimes \theta^{15}) \\  \label{curvature}
& & + dr \wedge \frac{\partial}{\partial r} (A-A_c)  
\end{eqnarray}
\end{lemma}
\begin{proof}
The curvature of an invariant connection $A= A^l_c + (A- A^l_c)$ is given by
\begin{equation}\label{curvatura}
F_A = F^l_{c} + d_{A^l_c} \left( A - A_c \right) + \frac{1}{2} \left[ \left( A -A^l_c \right) \wedge \left( A -A^l_c \right) \right],
\end{equation}
where $F^l_c$ is the curvature of the canonical invariant connection, computed in equation \ref{Canonicalcurvature}, and $d_{A^l_c} \left( A - A^l_c \right)$ is the covariant derivative of $A- A^l_c$ with respect to $A^l_c$. The statement that the connection is reducible follows from the Ambrose-Singer theorem, since the curvature always takes value in the $\mathfrak{u}(1) \subset \mathfrak{su}(2)$ generated by $T_1$.\\

For $l \neq 1$, the third therm in \ref{curvatura} is $A_1^2\theta^1 \wedge \theta^1 \otimes [T_1,T_1]$ and so vanishes. One is left with the computations of the second term, for which the Bianchi identity $d_{A_c^l}F_c^l =0$ can be used to conclude $d_{A_c^l} T_1=0$ and so
\begin{eqnarray}\nonumber
d_{A^l_c} (A-A^l_c) & = & d(A-A^l_c) + [A^l_c \wedge (A- A^l_c)] \\ \nonumber
& = & \dot{A}_1 dr \wedge \theta^1 \otimes T_1 + A_1 T_1 \otimes (\theta^{24} + \theta^{35} ).
\end{eqnarray}

The case $l=1$ is more involved. Using the Maurer-Cartan relations, the second term in \ref{curvatura} $I_2=d_{A_c} (A-A_c) = d(A-A_c) + [A_c \wedge (A- A_c)]$ is
\begin{eqnarray}\nonumber
d_{A_c} (A-A_c) & = & dr \wedge \frac{\partial}{\partial r} (A-A_c) + A_1 T_1 \otimes (\theta^{24} + \theta^{35} )\\ \nonumber
& & - (A_2 T_2 + A_3 T_3) \otimes \theta^{14} + (A_3 T_2 - A_2 T_3) \otimes \theta^{15} \\ \nonumber
& & + (A_4 T_2 + A_5 T_3 ) \otimes \theta^{12} + (-A_5 T_2 + A_4 T_3) \otimes \theta^{13} ,
\end{eqnarray}
where the vertical terms (i.e. those in $\mathfrak{h}$) from the exterior derivative have canceled with the ones coming from $[A_c \wedge (A- A_c)]$. The last term $I_3=\frac{1}{2} \left[ (A-A_c) \wedge (A-A_c) \right]$ is given by
\begin{eqnarray}\nonumber
I_3 & = & A_1 \theta^1 \wedge (A_2 \theta_2 - A_3 \theta^3 + A_4 \theta^4 - A_5 \theta^5) \otimes [T_1,T_2] \\ \nonumber
& & + A_1 \theta^1 \wedge (A_3 \theta_2 +A_2 \theta^3 + A_5 \theta^4 + A_4 \theta^5) \otimes [T_1,T_3] \\ \nonumber
& & + \left(A_2 \theta_2 - A_3 \theta^3 + A_4 \theta^4 - A_5 \theta^5 \right) \wedge \left( A_3 \theta_2 +A_2 \theta^3 + A_5 \theta^4 + A_4 \theta^5 \right) \otimes [T_2,T_3]  \\ \nonumber
& = & 2A_1 (A_2T_3-A_3T_2) \otimes \theta^{12}+ 2A_1 (A_4T_3 - A_5T_2) \otimes \theta^{14} \\ \nonumber
& &- 2A_1 (A_2T_2 + A_3T_3) \otimes \theta^{13} - 2A_1 (A_4T_2 + A_5T_3) \otimes \theta^{15} \\ \nonumber
& & + 2(A_2 A_5-A_4 A_3)T_1 \otimes ( \theta^{24} + \theta^{35} ) +  2(A_2 A_4 + A_5 A_3)T_1 \otimes (\theta^{25}- \theta^{34} )\\ \nonumber
& &  2(A_2^2 + A_3^2)T_1 \otimes \theta^{23} + 2 (A_4^2 + A_5^2 ) T_1\otimes \theta^{45}.
\end{eqnarray}
\end{proof}

\begin{lemma}\label{lem:CovD}
Let $\Phi \in \Omega^0(T^* \mathbb{S}^3, \mathfrak{g}_{P_l})$ be an invariant Higgs field and $A^l$ an invariant connection on $P_l$. Then, if $l \neq 1$, $\nabla_{A^l} \Phi = \dot{\phi} dr \otimes T_1$, while for $l=1$
\begin{eqnarray}\nonumber
\nabla_{A^1} \Phi & = &  \dot{\phi} dr \otimes T_1 \\ \nonumber
& &+ 2 \phi A_2 \left(  T_2 \otimes \theta^3 - T_3 \otimes \theta^2 \right) + 2 \phi A_3  \left(  T_2 \otimes \theta^2  + T_3 \otimes \theta^3 \right)  \\  \nonumber
& &- 2\phi A_4  (T_3 \otimes \theta^4 - T_2 \otimes \theta^5 )+ 2 \phi  A_5 ( T_2 \otimes \theta^4 +  T_3 \otimes \theta^5 ) .
\end{eqnarray}
\end{lemma}
\begin{proof}
This follows from computing $\nabla_{A^l} \Phi = \nabla_{A^l} \left( \phi  T_1 \right) = d \phi \otimes T_1 + \phi  \nabla_{A^l} T_1$. The first term is just $\dot{\phi} dr \otimes T_1$, while for the second term one uses that $A^l = A^l_c + (A^l - A^l_c)$, then $\nabla_{A^l} T_1 =d_{A^l_c} T_1 + [A^l- A^l_c, T_1]$, i.e.
\begin{eqnarray}\nonumber
\nabla_{A^l} \Phi =  \dot{\phi} dr \otimes T_1 + \phi  \left( d_{A^l_c} T_1 + [A^l- A^l_c, T_1] \right). \nonumber
\end{eqnarray}
Again, the Bianchi identity $d_{A^l_c} F_{A^l_c}=0$ for $A^l_c$ gives $d_{A^l_c} T_1 =0$ and one is left with the remaining terms. In the case $l \neq 1$ these vanish and $\nabla_{A^l} \Phi = \dot{\phi} dr \otimes T_1$, while for $l=1$ one has $[A^l- A^l_c, T_1] $ being equal to
$$2(A_3 \theta^2 + A_2 \theta^3 + A_5 \theta^4 + A_4 \theta^5) \otimes T_2 - 2 (A_2 \theta^2 - A_3 \theta^3 + A_4 \theta^4 - A_5 \theta^5) \otimes T_3.$$
The result follows.
\end{proof}

\subsection{Reduction to ODE's}\label{section:ODE}

This section uses the results from the previous section to reduce the Calabi-Yau monopole equations for invariant connections and Higgs fields to ODE's. The two cases $l=1$ and $l \neq 1$ are presented separately and the case $l =1$ ends up being the more important one. There is a more convenient way of writing the Calabi-Yau monopole equation
\begin{eqnarray}\nonumber
d_A \Phi_1 \wedge \frac{\omega^2}{2} + F_A  \wedge \Omega_2 =0 \ , \  F_A \wedge \frac{\omega^2}{2} = 0.
\end{eqnarray}
These follow from a straightforward computation for which the reader can consult proposition $3.1.3$ in \cite{Oliveira2014}. In that proposition, the complex monopole equations are rewritten in many equivalent ways.

\begin{proposition}\label{ODE}
Up to the action of a constant gauge transformation, $Spin(4)$ invariant Calabi-Yau monopoles on $P_l \rightarrow T^* \mathbb{S}^3 \backslash \mathbb{S}^3$ are in correspondence with solutions to the following set of ODE's. For $l \neq 1$,
\begin{eqnarray}\label{eq:A11}
\dot{A}_1 & = & -2 \frac{\dot{ \mathcal{G} }}{\mathcal{G}} A_1 \\ \label{eq:phi1}
\dot{\phi} & = & \frac{l}{4}\frac{r}{\mathcal{G}^2} \left( \frac{R_-}{R_+} -\frac{R_+}{R_-}\right).
\end{eqnarray}
While for $l=1$, the fields must satisfy the constraint $A_2 A_4 + A_3 A_5 =0$ and solve
\begin{eqnarray}\label{eq:A1}
\dot{A}_1 & = & -2 \frac{\dot{ \mathcal{G} }}{\mathcal{G}} \left(  A_1 +  2(A_2 A_5-A_4 A_3) \right) \\ \label{eq:phi}
\dot{\phi} & = & \frac{1}{\mathcal{G}^2} \left(  \frac{r}{4} \frac{R_-}{R_+} \left( 1-4(A_2^2 + A_3^2) \right) -\frac{r}{4} \frac{R_+}{R_-} \left( 1-4 (A_4^2 + A_5^2 ) \right) \right) \\ \label{eq:A2}
\dot{A}_2 & = &  - \frac{r}{2} \frac{1}{R^2_-} \left(A_2 + 2A_1 A_5 \right) - \frac{r}{\mathcal{G}} \phi A_2 \\ \label{eq:A3}
\dot{A}_3 & = &- \frac{r}{2} \frac{1}{R^2_-} \left(A_3 - 2A_1A_4 \right) - \frac{r}{\mathcal{G}} \phi A_3 \\ \label{eq:A4}
\dot{A}_4 & = &   - \frac{r}{2} \frac{1}{R^2_+} \left( A_4 - 2A_1 A_3  \right) +\frac{r}{\mathcal{G}} \phi A_4  \\ \label{eq:A5}
\dot{A}_5 & = &   - \frac{r}{2} \frac{1}{R^2_+} \left( A_5 + 2A_1A_2 \right) + \frac{r}{\mathcal{G}} \phi A_5 ,
\end{eqnarray}
with $\phi, A_i : (\epsilon , \infty) \rightarrow \mathbb{R}$, for $i=1,2,3,4,5$, $R_+= \sqrt{\frac{r^2+\epsilon^2}{2}}$, $R_-= \sqrt{\frac{r^2-\epsilon^2}{2}}$ and $\mathcal{G}=\sqrt{R_+ R_-} \mathcal{F}' (r^2)$, where $\mathcal{F}$ is the K\"ahler potential for the Stenzel metric and $\mathcal{F}'$ its derivative.
\end{proposition}
\begin{proof}
We shall work with the case $l=1$ and notice that there is no loss on generality in doing as we can always set $A_2=A_3=A_4=A_5=0$ in order to go to the case $l \neq 1$. To carry on the computations use the formulae \ref{kahlerform} and \ref{hol}, together with those computed in the previous section to evaluate the quantities, $\nabla_A \Phi \wedge \frac{\omega^2}{2}$, $F_A \wedge \omega^2$ and $F_A \wedge \Omega_2$. First of all we compute
\begin{eqnarray}\nonumber
F_A \wedge \omega^2 & = & -2 \mathcal{G}^2 dr \wedge \frac{\partial}{\partial r} (A- A_c) \wedge \theta^{2345} \\ \nonumber
& &  - 4 \mathcal{G} \dot{\mathcal{G}} \left( A_1 +  2(A_2 A_5-A_4 A_3)  \right) T_1 \otimes dr \wedge \theta^{12345} \\  \nonumber
& = &  2 \mathcal{G} \left(  \mathcal{G}\dot{A}_1 + 2 \dot{ \mathcal{G} } \left(  A_1 +  2(A_2 A_5-A_4 A_3) \right) \right)T_1 \otimes  dr \wedge \theta^{12345}.
\end{eqnarray}
Then setting $F_A \wedge \omega^2 =0$ gives the evolutions equation for $A_1$ which appears in the statement as equation \ref{eq:A11} for $l \neq 1$ and equation \ref{eq:A1} for $l =1$. The remaining equations follow from $-\nabla_A \Phi \wedge \frac{\omega^2}{2} = F \wedge \Omega_2=0$ and we now compute each side separately
\begin{eqnarray}\nonumber
\nabla_A \Phi \wedge \frac{\omega^2}{2} & = & - \mathcal{G}^2 \dot{\phi}\ T_1 \otimes dr \wedge \theta^{2345} \\ \nonumber
& & + 2\mathcal{G} \dot{\mathcal{G}} \phi \left( (A_3T_2-A_2 T_3) \otimes dr \wedge \theta^{1235} -  (A_2 T_2+A_3 T_3) \otimes dr \wedge \theta^{1234} \right) \\ \nonumber
& & - 2\mathcal{G} \dot{\mathcal{G}}  \phi  \left((A_5T_2-A_4 T_3) \otimes dr \wedge \theta^{1345} - (A_4 T_2 + A_5 T_3) \otimes dr \wedge \theta^{1245} \right). 
\end{eqnarray}
The computation of $F_A \wedge\Omega_2$ is long, but the outcome is
\begin{eqnarray}\nonumber
&&F_A \wedge  \Omega_2 = 4R_+ R_- (A_2 A_4 + A_3 A_5 )  T_1 \otimes \theta^{12345} \\ \nonumber
& &  + \left(  \frac{r}{4} \frac{R_-}{R_+} \left( 1-4(A_2^2 + A_3^2) \right) -\frac{r}{4} \frac{R_+}{R_-} \left( 1-4 (A_4^2 + A_5^2 ) \right) \right) T_1 \otimes dr \wedge \theta^{2345} \\ \nonumber
& & - R_-R_+ [  \dot{A}_2 T_2 + \dot{A}_3 T_3 +\frac{r}{2R_-^2} ( (A_2 + 2A_1 A_5)T_2 + (A_3 - 2A_1A_4)T_3  ) ]\otimes dr \wedge \theta^{1234} \\ \nonumber
& & - R_-R_+ [ -\dot{A}_3 T_2 + \dot{A}_2 T_3 + \frac{r}{2R_-^2}( (A_2 + 2A_1 A_5)T_3 - (A_3 - 2A_1A_4)T_2 ) ] \otimes dr \wedge \theta^{1235} \\ \nonumber
& & - R_-R_+  [ \dot{A}_4 T_1 + \dot{A}_5 j  + \frac{r}{2R_+^2} ( (A_4 - 2A_1 A_3)T_2 + (A_5 + 2A_1A_2)T_2 ) ] \otimes dr \wedge \theta^{1245} \\ \nonumber
& & - R_-R_+ [  -\dot{A}_5 T_2 + \dot{A}_4 T_3 + \frac{r}{2R_+^2} ( (A_4 - 2A_1 A_3)T_3 - (A_5 + 2A_1A_2)T_2 ) ] \otimes dr \wedge \theta^{1345}
\end{eqnarray}
Matching all these computations in $- \nabla_A \Phi \wedge \frac{\omega^2}{2} = F \wedge \Omega_2$ and using the identity $\frac{2\mathcal{G} \dot{\mathcal{G}}}{R_+ R_-} = \frac{r}{\mathcal{G}}$ (which is no more than the ODE assuring Stenzel's metric is Ricci flat) gives both: the constraint $4R_+ R_- (A_2 A_4 + A_3 A_5 )=0$ and the remaining equations in the statement.
\end{proof}

\begin{remark}
The Calabi-Yau monopole equations are overdetermined. In this specific example this can be directly seen from the ODE's in the statement of the previous proposition. In fact, for $l=1$ one sees that there are $6$ ODE's for $6$ real valued functions, but they are constrained to satisfy the identity $A_2A_4 + A_3 A_5=0$. Since the complex structure is integrable it is expected that the evolution encoded in the $6$ ODE's does preserve this constraint. In fact this will be shown later in proposition \ref{prop4}.
\end{remark}

\section{Calabi-Yau Monopoles on the Cone}\label{Conesolution}

This section studies Calabi-Yau monopoles on the Conifold. The most important point is the existence of an Abelian Calabi-Yau monopole given by the canonical invariant connection $A^l_c$. This is the pull back from $L^l \rightarrow D= \mathbb{P}^1 \times \mathbb{P}^1$ of a HYM connection. In fact, such HYM connections are the model for the asymptotic behavior of finite mass Calabi-Yau monopoles on general AC manifolds as shown in proposition $3.1.28$ of \cite{Oliveira2014}. Moreover, since $c_1(L^l) \in H^{1,1}(D, \mathbb{Z})$ is in the kernel of $ \cdot \cup [\omega_D]$, the existence of such a HYM connection can be proved by a short argument using Hodge theory, but in this section an explicit formula for such a connection is given.\\
Notice that $\mathfrak{g}_P \cong i \underline{\mathbb{R}} \oplus L^{2l}$, then using this decomposition let $\Phi = \phi \oplus 0$ with $\phi$ constant. Then $(A^l_c, \Phi)$ are Calabi-Yau monopoles on the Conifold and provide good asymptotic conditions for finite mass Calabi-Yau monopoles on $T^* \mathbb{S}^3$. In the system of ODE's this corresponds to taking $\phi$ constant and all the $A_i$'s to be zero. After writing the equations on the cone it will be trivial to see that this is indeed a solution. In fact a slightly more general result, proposition \ref{prop:MonopolesConifold}, classifying all ``constant" mass Calabi-Yau monopoles on the Conifold is obtained. Recall that the K\"ahler potential on the cone is $\mathcal{F}= \rho^2$, with $\rho=\left( \frac{3r}{2} \right)^{\frac{2}{3}} $, so we have
\begin{eqnarray}
\mathcal{F}'(r^2) = \left( \frac{3}{2} \right)^{\frac{1}{3}} r^{- \frac{2}{3}}= \frac{3}{2} \frac{1}{\rho}  \ , \ \mathcal{G} = \frac{1}{2} \left( \frac{3}{2} \right)^{\frac{1}{3}}  r^{\frac{4}{3}} = \frac{\rho^2}{2} \ , \ \dot{\mathcal{G}} =  \left( \frac{2}{3} \right)^{\frac{2}{3}} r^{\frac{1}{3}} = \frac{2}{3} \sqrt{\rho}. \nonumber
\end{eqnarray}
Substitute these in the equations from proposition \ref{ODE}, then for $l=1$ these turn into
\begin{eqnarray}\nonumber
\dot{A}_1 & = &- \frac{8}{3r} \left(  A_1 +  2(A_2 A_5-A_4 A_3) \right) \\ \nonumber
\dot{\phi} & = & 4 \left( \frac{2}{3} \right)^{\frac{2}{3}} r^{-\frac{5}{3}} \left(  (A_4^2 + A_5^2 ) -(A_2^2 + A_3^2) \right) \end{eqnarray}
together with the constraint $A_2 A_4 + A_3 A_5 =0$ and
\begin{eqnarray}\nonumber
\dot{A}_2 =  - \frac{1}{r}  \left(A_2 + 2A_1 A_5 \right) - 2 \left( \frac{2}{3r} \right)^{\frac{1}{3}}  \phi A_2 &  & \dot{A}_3 =  - \frac{1}{r} \left(A_3 - 2A_1A_4 \right) -   2 \left( \frac{2}{3r} \right)^{\frac{1}{3}} \phi A_3, \\ \nonumber
\dot{A}_4 =  -  \frac{1}{r} \left( A_4 - 2A_1 A_3  \right) + 2 \left( \frac{2}{3r} \right)^{\frac{1}{3}}  \phi A_4  &  & \dot{A}_5 =   -  \frac{1}{r} \left( A_5 + 2A_1A_2 \right) +  2 \left( \frac{2}{3r} \right)^{\frac{1}{3}}  \phi A_5. \nonumber
\end{eqnarray}
The following rescaling simplifies the equations and is a good preview of what will be done later for $T^* \mathbb{S}^3$. Define the fields $B_i$ via
$$B_2 = rA_2 \ \ , \ \ B_3 = r A_3 \ \ , \ \ B_4 =r A_4 \ \ , \ \ B_5=r A_5.$$
Use $\dot{A_i} + \frac{1}{r}A_i = \frac{1}{r}\dot{B}_i$, and change coordinates to $\rho$ via $\frac{d}{dr} = \left( \frac{2}{3r} \right)^{\frac{1}{3}} \frac{d}{d\rho}$ to obtain
\begin{eqnarray}\nonumber
\frac{dA_1}{d\rho} & = &- \frac{4}{\rho} A_1 +  \frac{18}{\rho^4}(B_2 B_5-B_4 B_3),  \\ \nonumber
\frac{d\phi}{d \rho} & = &  \frac{3^3}{2\rho^5} \left(  (B_4^2 +B_5^2 ) -(B_2^2 + B_3^2) \right),
\end{eqnarray}
together with the constraint $B_2B_4 + B_3B_5=0$ and
\begin{eqnarray}\nonumber
\frac{dB_2}{d \rho} =  -  \frac{3}{\rho}A_1 B_5 - 2 \phi B_2 &  & \frac{dB_3}{d \rho} = + \frac{3}{\rho}A_1B_4 -  2 \phi B_3 , \\ \nonumber
\frac{dB_4}{d \rho} =  + \frac{3}{\rho}A_1 B_3  + 2 \phi B_4 &  & \frac{dB_5}{d \rho} =  - \frac{3}{\rho}A_1B_2 +  2 \phi B_5 . \nonumber
\end{eqnarray}

\begin{proposition}\label{prop:MonopolesConifold}
For all $l \in \mathbb{Z}$ and in radial gauge, any $Spin(4)$-invariant Calabi-Yau monopole on $P_l$ over the Conifold with $\vert \Phi \vert$ a nonzero constant is given by
\begin{equation}\label{eq:ConeMonopole}
A^l = A^l_c + C \rho^{-4}  \theta^1 \otimes T_1  \ \ , \ \ \Phi = m T_1,
\end{equation}
with $C \in \mathbb{R}$ and $m \in \mathbb{R} \backslash \lbrace 0 \rbrace$. In particular, the canonical invariant connection $A^l_c$ is obtained by setting $C=0$ and in all cases $\vert A- A^l_c \vert = O(\rho^{-5})$ with all derivatives.
\end{proposition}
\begin{proof}
If $\vert \Phi \vert$ is constant, then $\phi=m \in \mathbb{R}$ and in a first case focus in the more involved case $l=1$. Make use of the extra gauge freedom and use $g \in U(1) \subset SU(2)$ to change the connection from $A-A_c^l$ to $g(A-A_c^l)g^{-1}$. This rotates $A_2 T_2 + A_3 T_3$ and $A_4 T_2 + A_5 T_3$ simultaneously. Hence, there is no loss of generality in supposing that $A_2=0$, i.e. $B_2=0$. Then, the constraint turns into $B_3B_5=0$, while the third equation is $A_1B_5=0$, then either $A_1=B_3=0$ or $B_5=0$.\\
First the case $A_1=B_3=0$, then in fact $A_2T_2 + A_3 T_3 =0$ and so the gauge freedom is still available to set $B_4=0$. Since $\phi=m$ the equation for $\frac{d \phi }{ d\rho}=0$ gives $B_5=0$ as well. So in this case $\Phi=m T_1$ and the connection is the canonical invariant one.\\
For the case where $B_5=0$, the second equation gives $B_4^2=B_3^2$, i.e. $B_3 = \pm B_4$. If one defines $B_1=\rho^4 A_1$, the remaining equations are
\begin{eqnarray}\label{eq:ode11}
\frac{dB_1}{d\rho} & = & \mp 18 B_4^2 \\ \label{eq:ode12}
\frac{d(B_4^2)}{d \rho} & = &  \pm \frac{3}{\rho^5}B_1  B_4^2 - 4m B_4^2 \\ 
\frac{d(B_4^2)}{d \rho} & = &  \pm \frac{3}{\rho^5}B_1  B_4^2 + 4m B_4^2.
\end{eqnarray}
Since $m \neq 0$ by hypothesis, the last two ODE's are compatible only in the case $B_4=0$ and so also $B_3=0$. One is left with solving the first equation which now says that $B_1$ is constant. The Calabi-Yau monopole to which this corresponds is given by the connection $A= A_c^1 + \frac{C}{\rho^4} \theta^1 \otimes T_1$ and the Higgs field $\Phi = m T_1$. Hence its is reducible and the connection is HYM and for $C=0$ is the canonical invariant one.\\

In the case $l \neq 1$, then immediately $B_2=B_3=B_4=B_5=0$ and the only equation is $\frac{dA_1}{d\rho} =- \frac{4}{\rho} A_1$. This can be integrated and gives \ref{eq:ConeMonopole}, which was obtained before for $l=1$. These monopoles do decay to the canonical invariant connection. However, this decay is at a polynomial rate, more specifically $\vert A- A_c^l \vert = O(\rho^{-5})$, which is due to the (unique) component which is "parallel" to the Higgs field. So if one imposes that the connection must decay faster than this rate the canonical invariant connection is the unique solution (setting $C=0$).
\end{proof}

\begin{remark}
All these Calabi-Yau monopoles are reducible and their connections are Hermitian Yang Mills (HYM) on the Conifold. The canonical invariant connection, obtained from $C=0$, is the unique one which is pulled back from the link. For $C \neq 0$ the connections differ from this one by $C \rho^{-4 } \theta^1= I d \left( \frac{3C}{8} \rho^{-4}\right) $, which is a harmonic $1$-form on the cone. In fact, notice that given an Abelian Calabi-Yau monopole $(A^0,\Phi^0)$ and a harmonic $1$-form $a$, then $(A^0+a, \Phi^0)$ is also a Calabi-Yau monopole.\\
Also, notice that it is also possible to solve the equations with $m=0$. Following the proof above the equations reduce to $\frac{dB_1}{d\rho} = \mp 18 B_4^2$ and $\frac{d(B_4^2)}{d \rho} =  \pm \frac{3}{\rho^5}B_1  B_4^2$. Integrating these with $B_4 \neq 0$ gives rise to $SU(2)$-irreducible HYM connections on the cone, which are not pulled back from $D=\mathbb{P}^1 \times \mathbb{P}^1$.
\end{remark}

\section{Reducible Calabi-Yau Monopoles in $T^* \mathbb{S}^3$}\label{lbig}

For reducible Calabi-Yau monopoles one must put all $A_i=0$, for $i \geq 2$. Then, only the first two equations in proposition \ref{ODE} survive. The first of them $\frac{dA_1}{d\rho} = -2 \frac{\dot{ \mathcal{G} }}{\mathcal{G}} A_1 $, can be readily integrated to give $A_1(r)= \frac{C}{\mathcal{G}^2}$, where $C \in \mathbb{R}$ is a constant. Regarding the second equation, using the function $h^2 = \frac{1}{\epsilon^2} R_+ R_-  \mathcal{G}$ and the radial coordinate $\rho$ gives $\frac{d\phi}{d\rho} = -\frac{l}{2h^2}$. This can be integrated to
$$\phi_l(\rho) = m - \int \frac{l}{2h^2(\rho)} d\rho,$$
with $m \in \mathbb{R}$. These monopoles have $\Phi$ being singular at the zero section (i.e. $\rho =0$), and by analogy with $3$ dimensions are called Dirac Calabi-Yau monopoles. 

\begin{definition}\label{def:Dirac}
Let $(X, \omega , \Omega)$ be a noncompact Calabi-Yau manifold and $N \subset X$ a special Lagrangian submanifold. A Dirac Calabi-Yau monopole is a Calabi-Yau monopole on a line bundle defined on the complement of $N$. $N$ will be called the singular set of the Calabi-Yau monopole.
\end{definition}

\begin{proposition}\label{prop:DiracMonS}
For all $l \in \mathbb{Z}$ and $C, m \in \mathbb{R}$, the connections and Higgs fields
$$A= A_c^l + \frac{C}{\mathcal{G}} \theta^1 \ \ , \ \ \phi = m - \int \frac{l}{2h^2(\rho)} d\rho,$$
are Dirac Calabi-Yau monopoles on $L^{\otimes l}$ for the Stenzel metric, with the zero section as singular set.
\end{proposition}

Their curvature is
\begin{equation}\label{Curvaturel}
F^l =  - \frac{l}{2}  (\theta^{23} + \theta^{45}) -2 C \frac{\dot{\mathcal{G}}}{\mathcal{G}^3} dr \wedge \theta^1 + \frac{C}{\mathcal{G}^2} (\theta^{24} + \theta^{35}) 
\end{equation}
Moreover, from the Appendix \ref{Appendix} one knows that $h(\rho) = \rho + O(\rho^3)$ for $\rho \ll 1$, while $h(\rho) = O (\rho^{5/2})$ for $\rho \gg 1$ and so
\begin{equation}\label{singularfi}
\phi(\rho) = \begin{cases}  \frac{1}{\rho} + O(\rho^{0}) &\mbox{if } \rho \ll 1 \\ 
m + \frac{cl}{\rho^4} + O(\rho^{-4-\epsilon}) & \mbox{if } \rho \gg 1 ,\end{cases}
\end{equation}
where $c >0$ is a constant independent of $l$ and only depending on $Vol_{g_{\Sigma}}(\Sigma)$ and $\epsilon>0$. In fact, using the formula \ref{RFmetric} for Stenzel's metric, we can check that $\phi$ is harmonic outside the zero section
\begin{eqnarray}\nonumber
\ast \Delta \phi = d \ast d \int_0^{\rho} \frac{1}{2h^2(s)} ds  = d \left(\frac{1}{2h^2(\rho)} \frac{\partial \rho}{\partial r} \ast dr \right) = d\epsilon^2 = 0. \nonumber
\end{eqnarray}

\section{Irreducible Calabi-Yau Monopoles in $T^* \mathbb{S}^3$}\label{lsmall}

This section reduces the system of ODE's in proposition \ref{ODE} to simpler ones and uses it to prove the main theorem \ref{Thm:TheoremStenzel}. This is done in a series of steps: first proposition \ref{newODE's} rescales the fields $A_i$ and changes coordinates from $s$ to $\rho$ in order to rewrite the ODE's. Then proposition \ref{prop4} rewrites the equations once again and shows the constraint $A_2 A_4 + A_3 A_5 =0$ is preserved by the evolution encoded in the other equations. Then we state and prove proposition \ref{lem:MainLem}, which contains much of the work we shall need in order to prove the main theorem \ref{Thm:TheoremStenzel}. More precisely, the proof of the main theorem requires splitting the analysis into $3$ cases. One of these cases uses proposition \ref{lem:MainLem} to reduce the problem to that of parameterizing spherically symmetric Bogomolnyi monopoles in $(\mathbb{R}^3, dr^2 + h^2(r) g_{\mathbb{S}^2})$, where  $h^2(\rho) = \frac{1}{\epsilon^2} R_+ R_-  \mathcal{G}$. The solution of this problem is given in the Appendix of \cite{Oliveira13} whose results are then used to conclude the proof of theorem \ref{Thm:TheoremStenzel}.

\begin{proposition}\label{newODE's}
Let the rescaled fields $B_i$ be defined via $B_1 = \mathcal{G}^2 A_1$, $B_2 = R_-A_2$, $B_3 = R_- A_3$, $ B_4 = R_+ A_4 $, $B_5= R_+ A_5$. Then, in terms of the distance function $\rho$, defined in \ref{s}, and using $h^2(\rho) = \frac{1}{\epsilon^2} R_+ R_-  \mathcal{G}$ the ODE's in proposition \ref{ODE} are given by the constraint $B_2B_4 + B_3B_5=0$ and
\begin{eqnarray}\nonumber
\frac{d\phi}{d\rho} & = &  -\frac{1}{2h^2(\rho)} \left( 1 - \frac{4}{\epsilon^2} \left( (B_4^2 + B_5^2) - (B_2^2 + B_3^2) \right) \right) \\  \nonumber
\frac{dB_1}{d\rho} & = &  -4 \left(B_2B_5 - B_4 B_3\right) \\ \nonumber
\frac{dB_2}{d\rho} & = &  - \frac{2}{\epsilon^2 h^2}B_1 B_5 - 2 \phi B_2 \\ \nonumber
\frac{dB_3}{d\rho} & = &    \frac{2}{\epsilon^2 h^2}B_1B_4  -2 \phi B_3 \\  \nonumber
\frac{dB_4}{d\rho} & = &    \frac{2}{\epsilon^2 h^2}B_1 B_3  +2 \phi B_4  \\  \nonumber
\frac{dB_5}{d\rho} & = &   -\frac{2}{\epsilon^2 h^2}B_1B_2 + 2 \phi B_5 . 
\end{eqnarray}
\end{proposition}
\begin{proof}
The constraint $B_2B_4 + B_3B_5=0$ is immediate from $A_2 A_4 + A_3 A_5 = 0$. Inserting the rescaled fields into the equation for $\dot{\phi}$ in proposition \ref{ODE} and rearranging gives
\begin{eqnarray}\nonumber
\dot{\phi} & = &  -\frac{r}{R_- R_+ \mathcal{G}^2} \frac{\epsilon^2}{4} \left( 1 - \frac{4}{\epsilon^2} \left( (B_4^2 + B_5^2) - (B_2^2 + B_3^2) \right) \right) 
\end{eqnarray}
Next use $\frac{d}{dr} = \frac{r}{2 \mathcal{G}} \frac{d}{d\rho}$ to change coordinates to $\rho$ and $h^2= \frac{1}{\epsilon^2} R_+ R_-  \mathcal{G}$ to obtain the equation in the statement for $\frac{d \phi }{d \rho}$. To analyze the other equations use $\dot{R_+}=\frac{r}{2R_+}$ and $\dot{R_-}= \frac{r}{2R_-}$ to compute
$$\dot{B_i} = R_- \left( \dot{A_i} + \frac{r}{2R_-^2} A_i \right)\ , \  \dot{B_j} = R_+ \left( \dot{A_j} + \frac{r}{2R_+^2} A_j \right)$$
for $i=2,3$ and $j=4,5$. Inserting the equations in proposition \ref{ODE} into these, gives $\dot{B}_2 =  -  \frac{r}{R_+R_-} A_1 B_5 - \frac{r}{\mathcal{G}} \phi B_2$, $\dot{B}_3 =  \frac{r}{R_+R_-} A_1B_4  - \frac{r}{\mathcal{G}} \phi B_3$, $\dot{B}_4 =  \frac{r}{R_+R_-} A_1 B_3  +\frac{r}{\mathcal{G}} \phi B_4 $ and $\dot{B}_5 = - \frac{r}{R_+R_-}A_1B_2 + \frac{r}{\mathcal{G}} \phi B_5$. Once again, we change coordinates to $\rho$ and these equations turn into
\begin{eqnarray}\nonumber
\frac{dB_2}{d\rho} =  - \frac{2 \mathcal{G}}{R_- R_+} A_1 B_5 - 2 \phi B_2 & , & \frac{dB_3}{d\rho} =   \frac{2 \mathcal{G}}{R_- R_+} A_1B_4  - 2 \phi B_3 \\ \nonumber
\frac{dB_4}{d\rho} =    \frac{2 \mathcal{G}}{R_- R_+} A_1 B_3  +2 \phi B_4  & , & \frac{dB_5}{d\rho} =   -  \frac{2 \mathcal{G}}{R_- R_+} A_1B_2 + 2 \phi B_5 , \nonumber
\end{eqnarray}
and now changing from $A_1$ to $B_1 = \mathcal{G}A_1$ and using $h^2= \frac{1}{\epsilon^2} R_+ R_-  \mathcal{G}$, gives the last four equations in the statement. To obtain the remaining equation multiply the equation containing $\dot{A}_1$ in proposition \ref{ODE} by $\frac{2 \mathcal{G}}{r}$ in order to ease the coordinate change. This gives
$$\frac{dA_1}{d\rho} = -\frac{4\dot{\mathcal{G}}}{r} A_1 - \frac{2\dot{\mathcal{G}}}{rR_+R_-} 4 \left(B_2B_5 - B_4 B_3\right).$$
Multiply this equation by $\mathcal{G}^2$ and pass the terms having $A_1$ to the same side, then this term of the equation turns into $\mathcal{G}^2 \frac{dA_1}{d\rho} + \frac{4\mathcal{G}^2}{r} \frac{r}{2 \mathcal{G}} \frac{d\mathcal{G}}{d\rho} A_1= \mathcal{G}^2 \frac{dA_1}{d\rho} + 2\mathcal{G}\frac{d\mathcal{G}}{d\rho} A_1$, which is precisely $\frac{d}{d\rho} \left( \mathcal{G}^2 A_1 \right)$ and replaced back into the equation gives
$$\frac{dB_1}{d\rho} =  - \frac{2 \mathcal{G}^2 \dot{\mathcal{G}} }{rR_+R_-} 4\left(B_2B_5 - B_4 B_3\right).$$
Next recall from remark \ref{normalization} that the ODE reduction of the Monge-Amp\`ere equation is $2 \mathcal{G}^2 \dot{\mathcal{G}} = rR_+ R_-$. Hence this equation also turns into the one in the statement.
\end{proof}

The following is a general lemma on certain systems of ODE's, which will prove to be useful in order to analyze the consistence of the Calabi-Yau monopole equations in proposition \ref{prop4} below.

\begin{lemma}\label{lem:Integrability}
Let $A_1(r),A_2(r),B_1(r),B_2(r)$ be real valued functions and $f(r),g(r)$ complex valued functions, such that $\Re(f \overline{g})=0$ at $r=r_0 \in \mathbb{R}$. Suppose $f$ and $g$ are subject to the following ODE's
\begin{eqnarray}\nonumber
\dot{g} = A_1 g + i B_1 f \ , \ \dot{f} = A_2 f + i B_2 g. 
\end{eqnarray}
If $\Re(f \overline{g})=0$ at $r = r_0 \in \mathbb{R}$, then $\Re(f \overline{g})=0$ for all $r \in \mathbb{R}$ and both phases $\chi_1, \chi_2$ of $f,g$ are constant. Moreover, for $fg \neq 0$ these satisfy $\chi_2 - \chi_1 = \frac{\pi}{2} + \pi k$, for some $k \in \mathbb{Z}$.
\end{lemma}
\begin{proof}
The fact that $\Re(f \overline{g}) =0$ is preserved by the flow follows from computing
\begin{eqnarray}\nonumber
\frac{d}{dr} (f \overline{g} )& = & \dot{f} \overline{g} + f \dot{\overline{g}} = ( A_2 f + i B_2 g ) \overline{g} + f ( A_1 \overline{g} - i B_1 \overline{f}) \\ \nonumber
& = & (A_1 + A_2 ) f\overline{g} + i  ( B_2 \vert g \vert^2 - B_1 \vert f \vert^2). 
\end{eqnarray}
So $\frac{d}{dr} \Re (f \overline{g})= (A_1 + A_2) \Re(f \overline{g})$, so that in general $\Re (f \overline{g}) = ke^{\int A_1 + A_2}$ and if at $r_0$ this vanishes then $\Re(f \overline{g})=0$ always. If both $f,g \neq 0$ and $0 = \Re(f \overline{g})=r_1r_2 \Re( e^{i \left( \chi_1-\chi_2 \right)})$, then one needs $e^{i \left( \chi_1-\chi_2 \right)}$ to be purely imaginary, i.e. $\chi_2 - \chi_1 = \frac{\pi}{2} + \pi k$ for some $k \in \mathbb{Z}$. To see that also each phase is constant let $f= r_1 e^{i \chi_1}$ and $g= r_2 e^{i\chi_2}$, then the second equation is
\begin{eqnarray}\nonumber
\dot{r_1}e^{i \chi_1}+ \dot{\chi}_1 e^{i\left( \chi_1 + \frac{\pi}{2} \right) }& = & A_2 r_1 e^{i \chi_1} + B_2 r_2 e^{i \left(\frac{\pi}{2} + \chi_1 \pm \frac{\pi}{2} \right)} = \left( A_2 r_1 \pm  B_2 r_2 \right) e^{i \chi_1}.
\end{eqnarray}
So as a result one has $\dot{\chi}_1=0$ and since the phase difference is constant also $\dot{\chi}_2=0$.
\end{proof}

\begin{proposition}\label{prop4}
Let $f_1, f_2 : X \rightarrow \mathbb{C}$ be given by $f_1 = B_2 + iB_3$, $f_2 = B_4 + i B_5$ and denote their phases by $\chi_1, \chi_2$ respectively. The constraint in theorem \ref{newODE's} is equivalent to $\Re(f_1 \overline{f_2})=0$ and if initially satisfied, is preserved by the other equations which are
\begin{eqnarray} \nonumber
\frac{d\phi}{d\rho} & = &  -\frac{1}{2h^2(s)} \left( 1 - \frac{4}{\epsilon^2} \left( \vert f_2 \vert^2 - \vert f_1 \vert^2 \right) \right) \\ \nonumber
\frac{dB_1}{d\rho} & = & 4 \Im(f_1 \overline{f_2}) \\ \nonumber
\frac{df_1}{d\rho} & = &  \frac{2i}{\epsilon^2 h^2} B_1 f_2 - 2 \phi f_1 \\ \nonumber
\frac{df_2}{d\rho} & = &   -\frac{2i}{\epsilon^2 h^2} B_1 f_1 + 2 \phi f_2 .
\end{eqnarray}
Moreover, the phases $\chi_1,\chi_2$ are constant and if $f_1f_2 \neq 0$, then $\chi_2 - \chi_1 = \frac{\pi}{2} + \pi k$, for some $k \in \mathbb{Z}$.
\end{proposition}
\begin{proof}
The evolution equation for $B_1$ and the constraint are obtained by using $\Re(f_1 \overline{f_2})=B_2B_4 + B_3 B_5$ and $\Im(f_1 \overline{f_2})= B_3 B_4 - B_2B_5$. The other equations follow from computing
\begin{eqnarray} \nonumber
\frac{df_1}{d\rho} & = & \frac{2}{\epsilon^2 h^2} A_1 (-B_5 + i B_4) - 2 \phi (B_2+iB_3) \\ \nonumber
& = &  \frac{2i}{\epsilon^2 h^2}  B_1 f_2 - 2 \phi f_1,
\end{eqnarray}
and similarly for $f_2$. To obtain the first equation, just notice $\frac{4}{\epsilon^2} \left( (B_4^2 + B_5^2)- (B_2^2 + B_3^2) \right) =  \frac{4}{\epsilon^2} \left( \vert f_2 \vert^2 - \vert f_1 \vert^2 \right) $. The proof that the constraint $\Re(f_1 \overline{f_2})=0$ is preserved by the motion and the statement regarding the phases is a direct application of lemma \ref{lem:Integrability} above.
\end{proof}

The next result will be central in the proof of the main theorem. In order for the statement not to seem mysterious we shall now do a short preview of the situation we will encounter during that proof.\\
To tackle the equations in proposition \ref{prop4} it will be useful to split into the cases $f_1f_2 =0$ and $f_1 f_2 \neq 0$. In the second case $f_1 f_2 \neq 0$ and so as stated in lemma \ref{lem:Integrability}, the phases $\chi_1, \chi_2$ are constant and $\chi_1-\chi_2 = \frac{\pi}{2} + \pi k$. One can then use an invariant constant gauge transformation, in order to have $\chi_1=\frac{\pi}{2}, \chi_2 =-\pi k$, which gives $f_1=iB_3$ and $f_2=(-1)^k B_4$. One must remark that the initial conditions in equation \ref{extendingA} in the statement, are those which are required for the connection to extend over the zero section.

\begin{proposition}\label{lem:MainLem}
Let $(\phi, B_1, B_3, B_4)$ a be solution to the equations 
\begin{eqnarray}\label{best}
\frac{d\phi}{d\rho} & = &  -\frac{1}{2h^2(s)} \left( 1 - \frac{4}{\epsilon^2} \left( B_4^2 - B_3^2 \right) \right) \\ \label{best2}
\frac{dB_1}{d\rho} & = &  4 (-1)^k B_3 B_4 \\ \label{best3}
\frac{dB_3}{d\rho} & = &    2\frac{(-1)^k}{\epsilon^2 h^2}B_1 B_4 -2 \phi B_3 \\ \label{best4}
\frac{dB_4}{d\rho} & = &   2 \frac{(-1)^k}{\epsilon^2 h^2}B_1 B_3  + 2 \phi B_4,
\end{eqnarray}
such that for in a neighborhood of $\rho=0$
\begin{equation}\label{extendingA}
B_1(\rho)= O(\rho^3) \ \ , \ \ B_3 (\rho) = O(\rho) \ \ , \ \ B_4(\rho) = \frac{\epsilon}{2} +O(\rho^2).
\end{equation}
Then $B_1=B_3=0$, $B_4= \frac{2}{\epsilon}a$ and $(a,\phi)$ must satisfy the equations
\begin{eqnarray}\label{tudo}
\frac{d\phi}{d\rho} & = & -\frac{1}{2h^2(\rho)} \left( 1-a^2\right) \\ \label{tudo2}
\frac{da}{d\rho} & = & 2\phi a,
\end{eqnarray}
subject to the conditions that $a(0)=1$ and $\phi(0)=0$.
\end{proposition}
\begin{proof}
One must find all the possible solutions $\phi,B_1,B_3,B_4$ to the system in the statement constrained so that \ref{extendingA} holds. Notice that a possible solution is given by taking $B_1=B_3=0$, $B_4= \frac{2}{\epsilon}a$ and $(a,\phi)$ solving the system \ref{tudo}, \ref{tudo2} with the conditions that $a(0)=1$ and $\phi(0)=0$. These conditions together with the equations do guarantee \ref{extendingA}. The proof is then reduced to showing that these are all the solutions. To do this use equations \ref{best2}, \ref{best3} and \ref{best4} and compute
\begin{eqnarray}\nonumber
\frac{d^2 B_1}{d\rho^2} & = &  4 (-1)^k \left( \frac{d  B_3}{d \rho} B_4 + B_3 \frac{d B_4}{d \rho} \right) \\ \nonumber
& = & 4 (-1)^k \left( 2\frac{(-1)^k}{\epsilon^2 h^2}B_1 \left( B_4^2 + B_3^2 \right) +  2 \phi (B_4B_3 - B_3B_4) \right)\\ \nonumber
& = & \frac{2u}{h^2}B_1 ,
\end{eqnarray}
where $u= \frac{4}{\epsilon^2} \left( B_3^2 + B_4^2 \right)$. Lemma \ref{hbehaviour} in Appendix \ref{Appendix} shows that for $\rho$ close to $0$, $h^2(\rho)=\rho^2 \psi(\rho)$, where $\psi(\rho)$ is real analytic with $\psi(0)=1$. Then the solutions must be real analytic and one can write
$$\frac{2u}{h^2}= \rho^{-2} \sum_{j=0}^{+ \infty} \varphi_j \rho^{j} \ \ , \ \ B_1(\rho) = \sum_{k=0}^{+ \infty} b_k \rho^k,$$
for some $\varphi_j, b_k$, with $\varphi_0 \neq 0$. Recall the hypothesis that $B_1(\rho)= O(\rho^3)$, this implies $b_0=b_1=b_2=0$. Inserting the series above into $\frac{d^2 B_1}{d\rho^2} = \frac{2u}{h^2}B_1$, just using that $b_0=b_1=0$ and rearranging gives
$$\sum_{i=0}^{+\infty} (i+2)(i+1) b_{i+2} \rho^i = \sum_{i=0}^{+\infty} \left( \sum_{0 \leq j \leq i} \varphi_j b_{i-j+2} \right) \rho^i ,$$
so one can use this to get the recurrence relation
$$b_{i+2}= \frac{1}{(i+1)(i+2)- \varphi_0} \sum_{0<j \leq i} \varphi_j  b_{i+2-j},$$
with $b_0=b_1=0$. This recurrence relation is completely determined by $b_2$, which vanishes by hypothesis ($B_1(\rho)= O(\rho^3)$). Hence, all the $b_i$'s vanish by the recurrence relation above and so $B_1=0$.\\
We now use the fact that $B_1=0$ to finish the proof. First, notice that from $B_1=0$ it follows from equation \ref{best2} that $B_3B_4=0$. So one must have $B_3=0$ as $B_4=0$ would contradict the hypothesis that $B_4(0)= \frac{\epsilon}{2}$, which then reduces the system to the one in the statement. The initial conditions $\phi(0)=0$ and $a(0)=1$ together with the equations do guarantee that \ref{extendingA} holds because \ref{tudo2} implies that $\dot{a}(0) = 2a(0)\phi(0)=0$.
\end{proof}

Equipped with the Appendix of \cite{Oliveira13} we are now in position of proving the main theorem \ref{Thm:TheoremStenzel} regarding Calabi-Yau monopoles for the Stenzel metric in $T^* \mathbb{S}^3$.

\subsection{Proof of the main theorem \ref{Thm:TheoremStenzel}}

Start from the equations as stated in proposition \ref{prop4}, then the phases $\chi_1,\chi_2$ are constant and
$$\Re(f_1 \overline{f_2})= \vert f_1 \vert \vert f_2 \vert \Re( e^{i(\chi_1-\chi_2)}).$$
This quantity vanishes if and only if either $\vert f_1 \vert=0$, or $\vert f_2 \vert=0$, or $\chi_1-\chi_2 = \frac{\pi}{2} + \pi k$ for some $k \in \mathbb{Z}$. Before proceeding with the case splitting, notice that for the connection to be asymptotic to the canonical invariant connection (which is HYM on the cone) one must have all $A_i$'s converging to $0$. This implies that the $B_i$'s must grow at most at a polynomial rate. Moreover, recall from remark \ref{RemarkGauge} that one can still use an invariant constant gauge transformation, i.e. $g \in U(1) \subset SU(2)$ which rotates $A-A_c^1$ to $g(A - A_c^1)g^{-1}$. This rotates the phases $\chi_1, \chi_2$ simultaneously and will be used in different ways in each of the different cases below

\begin{enumerate}

\item If $f_1=0$, the equations imply $\chi_2$ is constant and so a constant gauge transformation can be used to make $\chi_2=0$ so that $f_2=B_4$ is real. Then, the equations from proposition \ref{prop4} give that $B_1 B_4 =0$, $\frac{dB_1}{d\rho}=0$ and
\begin{eqnarray}\nonumber
\frac{d\phi}{d\rho} =  \frac{1}{2h^2} \left(  \frac{4}{\epsilon^2}  B_4^2  -1\right) \ , \frac{dB_4}{d\rho} = 2 \phi B_4 .
\end{eqnarray}
The conditions that the connection which a possible solution encodes extends over the zero section are studied in the Appendix \ref{Appendix}. It is shown in lemma \ref{lem:AextendingA} of that Appendix that for the connection to extend one needs $B_1(\rho)= O(\rho^3) $, $ B_3 (\rho) = O(\rho)$ and $B_4(\rho) = \frac{\epsilon}{2} +O(\rho^2)$, for $\rho$ close to $0$. From the equations one knows that $B_1$ must be constant and so vanish in order to satisfy the initial condition. Setting $a= \frac{2}{\epsilon}B_4$, the equations reduce to
\begin{eqnarray}\nonumber
\frac{d\phi}{d\rho} =  \frac{1}{2h^2} \left(  a^2  -1\right) \ , \  \frac{da}{d\rho} = 2 \phi a 
\end{eqnarray}
Together with the conditions that $a(0)=1$ and $\phi(0)=0$, which do imply (using the second equation) $a(\rho)=1+O(\rho^2)$ and so $B_4(\rho) = \frac{\epsilon}{2} +O(\rho^2)$. Notice that this is the system describing invariant monopoles in $\mathbb{R}^3$ equipped with the metric $d\rho^2 + h^2(\rho)g_{\mathbb{S}^2}$.

\item The case $\vert f_2 \vert=0$ is excluded as the condition that $B_4^2(0) = \frac{\epsilon}{2}$ can not be satisfied and the connection would not extend smoothly through the zero section.

\item The last case is when $f_1 f_2 \neq 0$ and $\chi_1-\chi_2 = \frac{\pi}{2} + \pi k$ and the phases are constant. As above, one can then use an invariant constant gauge transformation, to make $\chi_1=\frac{\pi}{2}, \chi_2 =-\pi k$, which gives $f_1=iB_3$ and $f_2=(-1)^k B_4$. The Calabi-Yau monopole equations are
\begin{eqnarray}\nonumber
\frac{d\phi}{d\rho} & = &  -\frac{1}{2h^2(s)} \left( 1 - \frac{4}{\epsilon^2} \left( B_4^2 - B_3^2 \right) \right) \\  \nonumber
\frac{dB_1}{d\rho} & = &  4 (-1)^k B_3 B_4 \\ \nonumber
\frac{dB_3}{d\rho} & = &    2\frac{(-1)^k}{\epsilon^2 h^2}B_1 B_4 -2 \phi B_3 \\ \nonumber
\frac{dB_4}{d\rho} & = &   2 \frac{(-1)^k}{\epsilon^2 h^2}B_1 B_3  + 2 \phi B_4,
\end{eqnarray}
subject to the conditions so that the connection extends smoothly over the zero section as shown in lemma \ref{lem:AextendingA} in the Appendix \ref{Appendix}. This is precisely the system analyzed in proposition \ref{lem:MainLem} and once again the problem has been reduced to the one of solving the ODE's for invariant monopoles in $\mathbb{R}^3$.
\end{enumerate}

Putting aside the second case where there are no smooth solutions, we have been reduced to the problem parameterizing invariant monopoles on $\mathbb{R}^3$. More precisely, the ODE problem in the conclusion to proposition \ref{lem:MainLem} together with the condition that there is $k \in \mathbb{Z}$ such that $\lim_{\rho \rightarrow + \infty} \rho^{-k}a =0$, where $a = \frac{\epsilon}{a}B_4$. As already mentioned before this is precisely the system solved in the Appendix of \cite{Oliveira13} and rest of the proof stands on invoking the results therein. The first item in the main result of that Appendix states that any solution $(a, \phi)$ has a well-defined finite limit 
$$\lim_{\rho \rightarrow \infty} \phi(\rho)  \in \mathbb{R}^{-},$$
and that for each value of $m \in \mathbb{R}^{-}$ there is one and only one solution. Hence, such value parametrizes the moduli space of invariant Calabi-Yau monopoles and this proves the first item in theorem \ref{Thm:TheoremStenzel}.\\
For the proof of the second and third statements, a preliminary digression is needed. Let $(a_m,\phi_m)$ be the solution associated with the value $m$, i.e. with $\phi_m$ converging to $m \in \mathbb{R}^-$. This corresponds to the Calabi-Yau monopole with $B_1=B_2=B_3=B_5=0$, $B_4= \frac{\epsilon}{2} a_m$ and $\phi= \phi_m$, which can be written
\begin{equation}\label{eq:MonopoleM}
 A_m= A_c^1 + \frac{\epsilon  }{2} \frac{a_m}{R_+} \left(  \theta^4 \otimes T_2 + \theta^5 \otimes T_3 \right) \ \ , \ \ \Phi_m =\phi_{m} T_1 .
\end{equation}
We would like to directly apply the results in the second and third items of the main result of the Appendix to \cite{Oliveira13} to the restriction of \ref{eq:MonopoleM} to the $\mathbb{R}^3$ fibres of $T^* \mathbb{S}^3 \rightarrow \mathbb{S}^3$. The problem, is that those results would only apply for monopoles on the $\mathbb{R}^3$ fibres normal to the zero section equipped with the spherically symmetric metric $h= d\rho^2 + h^2(\rho) g_{\mathbb{S}^2}$. These later $3$-dimensional monopoles on the fibers can be written
\begin{equation}\label{eq:MonopoleFake}
 \tilde{A}_m= A_c^1 + \frac{a_m}{2} \left(  \theta^4 \otimes T_2 + \theta^5 \otimes T_3 \right) \ \ , \ \ \tilde{\Phi}_m =\phi_{m} T_1.
\end{equation}
We shall now use the results in \cite{Oliveira13} for these in order to prove the corresponding statement for the genuine Calabi-Yau monopole \ref{eq:MonopoleM}. The two Higgs fields are the same $\tilde{\Phi}_{\lambda} = \Phi_{\lambda}$ so we shall now focus on the connections. For the proof of the second item one needs to show that for all $R, \delta >0$ there are $m$ and $\eta(R,\delta, m)>0$ such that $\Vert s_{\eta}^* A_{m} -  A^{BPS} \Vert_{C^{0}( B_{R})}  \leq \delta$. Let $s_{\eta}= \exp_{\eta}$ be the exponential in the fiber directions, then
\begin{eqnarray}\nonumber
\Vert s_{\eta}^* A_{m} -  A^{BPS} \Vert_{C^{0}( B_{R})} & \leq &  \Vert s_{\eta}^*\tilde{A}_{m} -  A^{BPS} \Vert_{C^{0}( B_{R})} +  \Vert  s_{\eta}^*\tilde{A}_{m} - s_{\eta}^*A_{m} \Vert_{C^{0}( B_{R})}
\end{eqnarray}
and use the corresponding statement in second item of the main result in the Appendix to \cite{Oliveira13}. This guarantees the first term can be made as small as one wishes, i.e. there is $\eta'>0$ such that the first term is less than $\frac{\delta}{2}$. Regarding the second term
\begin{eqnarray}\nonumber
\Vert  s_{\eta}^*\tilde{A}_{m} - s_{\eta}^*A_{m} \Vert_{C^{0}( B_{R})} & = &    \Vert  \tilde{A}_{m} - A_{m} \Vert_{C^{0}( B_{\eta R})}\\ \nonumber
& \leq &  \sup_{\rho \leq \eta R}  \Big\vert \left( a_{m} \left( 1- \frac{\epsilon}{ R_+}  \right) \right) \vert \theta_4 \vert_{g_E}  \Big\vert \\ \nonumber
& \leq &  \sup_{\rho \leq \eta R}  \Big\vert(   a_{m} \left(\frac{\rho^2}{2 \epsilon^{-4/3}} \right) \frac{1}{\rho} \Big\vert \leq \frac{ \eta R}{4 \epsilon^{4/3}} ,\nonumber
\end{eqnarray}
where in the last line one uses the fact that $R_+=\epsilon + \frac{1}{2 \epsilon^{1/3}} \rho^2 + ...$. Hence the estimate
$$\Vert s_{\eta}^* A_{\lambda} -  A^{BPS} \Vert_{C^{0}( B_{R})}  \leq \delta,$$
follows by making $\eta$ equal to the minimum of $\eta'$ and $\delta \frac{2 \epsilon^{4/3}}{R}$.\\
Notice that $A_m-A^1_c$ and $\tilde{A}_m-A^1_c$ differ by a factor of $\frac{\epsilon}{ R_+}$. Since, this is bounded and independent of $m$, the third item statement of theorem \ref{Thm:TheoremStenzel} follows directly from applying the third item in main result of the Appendix to \cite{Oliveira13}.

\begin{remark}
In the same gauge used so far, the curvature of $A_m$ is
\begin{eqnarray}\nonumber
F_{A_m} & = &\left( \left( \frac{\epsilon a_{m}}{R_+} -1 \right)^2  \theta^{45}  - \theta^{23} \right) \otimes \frac{T_1}{2} + \frac{\epsilon a_m}{2R_+} \left( \theta^{12} \otimes T_2 + \theta^{13}  \otimes T_3 \right)  \\ \nonumber
& & + \frac{d}{dr} \left( \frac{\epsilon  a_{m}}{2R_+} \right) \left(  dr \wedge \theta^4 \otimes T_2 + dr \wedge \theta^5 \otimes T_3 \right) .
\end{eqnarray}
Since the functions $a_m$ decay exponentially with $\rho$, the connection $A_m$ is exponentially asymptotic to the canonical invariant connection $A_c^1$. In fact the monopole $(A_m, \Phi_m)$ on $P_1$ with mass $m$ approaches the corresponding mass $m$ Dirac monopole from proposition \ref{prop:DiracMonS}. In particular, we can use the fact that this later one is explicit to compute the intermediate energy as defined in \cite{Oliveira2014}
$$E_I(A_m , \Phi_m)= \frac{m \epsilon^2}{2} \int_{\mathbb{S}^3 \times \mathbb{S}^2} \theta^{12345}.$$
Let $L$ be the complex line bundle over the Sasaki-Einstein $\mathbb{S}^3 \times \mathbb{S}^2$ defined in remark \ref{rem:DefL}, then its first Chern class $c_1(L) \in H^2(\mathbb{S}^3 \times \mathbb{S}^2, \mathbb{R})$ is a monopole class as in definition $3.2.3$ of \cite{Oliveira2014}. Moreover, the formula above for the intermediate energy can be matched up with corollary $3.1.25$ of that reference, which states that
$$E_I(A_m , \Phi_m)= 4 \pi m  \int_{\mathbb{S}^3 \times \mathbb{S}^2} c_1(L) \cup [i^* \Omega_1].$$
\end{remark}

\begin{remark}
During the proof there were some cases whose analysis was excluded as they did not satisfy the necessary conditions for the connection to extend over the zero section (see lemma \ref{lem:AextendingA} in the Appendix \ref{Appendix}). However in some cases Calabi-Yau monopoles with singularities are possible

\begin{enumerate}
\item In the first case with $f_1=0$ one can also take $f_2=0$ in order to solve the equations. Then, $B_1$ is constant, $\frac{d\phi}{d \rho}= - \frac{1}{2h^2}$ and the only solutions are reducible to one of the Dirac Calabi-Yau monopoles in proposition \ref{prop:DiracMonS}, i.e. $A= A_c^1 + \frac{C}{\mathcal{G}^2} \theta^1 \otimes T_1$ and $\Phi = \left( m - \int \frac{1}{2h^2(\rho)} d \rho \right) \otimes T_1$.

\item In the case $f_1 \neq 0$ but $\vert f_2 \vert=0$, and using the gauge in which $f_1=iB_3$, the system in proposition \ref{prop4} reduces to $B_1 B_3 = \frac{dB_1}{d\rho} = 0$ and
\begin{eqnarray}\nonumber
\frac{d\phi}{d\rho} =  -\frac{1}{2h^2(s)} \left( 1 + B_3^2 \right) \ , \ \frac{dB_3}{d\rho} =  -2 \phi B_3.
\end{eqnarray}
So $B_1$ is constant and either $B_3=0$ or $B_1=0$. If $B_3=0$ the unique solutions are the Dirac Calabi-Yau monopole from the previous case. If $B_1=0$, then there are no smooth solutions as well since $1 + B_3^2 >0 $ and $h(\rho)= O(\rho)$ for $\rho \ll 1$, also the Higgs field is unbounded at the zero section. So any possible solution will give rise to irreducible Calabi-Yau monopoles with a Dirac type singularity at the zero section.
\end{enumerate}
\end{remark}

\subsection{Explicit Hermitian Yang Mills $SU(2)$ Connection}\label{sec:HYM}

\begin{theorem}\label{t1}
There is an irreducible Hermitian Yang Mills connection on $P_1 \rightarrow T^* \mathbb{S}^3$ for Stenzel's Calabi-Yau structure. In the same gauge used before, it is given by
\begin{eqnarray}
A & = & A_c^1 + \frac{\epsilon}{2R_+} \left( \theta^4 \otimes T_2 + \theta^5 \otimes T_3 \right),
\end{eqnarray}
and its curvature by
\begin{eqnarray}\nonumber
F_A & = &  - \frac{1}{2} \left( \theta^{23} +\frac{R_-^2}{R_+^2} \theta^{45} \right) \otimes T_1 + \frac{\epsilon}{2R_+} \left(  T_2 \otimes \theta^{12} + T_3 \otimes \theta^{13} \right) \\ \nonumber
& & - \frac{\epsilon}{4} \frac{r}{R_+^3} \left( T_2 \otimes dr \wedge \theta^4 + T_3 \otimes dr \wedge \theta^5 \right).
\end{eqnarray}
\end{theorem}
\begin{proof}
This solution is obtained by setting $a=1$ and $\phi=0$, i.e. $B_1=B_3=0$ and $B_4=\frac{\epsilon}{2}$. These satisfy the conditions from lemma \ref{lem:AextendingA} in the Appendix \ref{Appendix}, so the resulting connection extends over the zero section, is irreducible and HYM. For this solution $A_4=\frac{\epsilon}{2R_+}$ and $\dot{A_4}= - \frac{\epsilon}{4} \frac{r}{R_+^3}$, so using the formula \ref{curvature} one can compute the curvature as in the statement.
\end{proof}

\begin{remark}
$A \rightarrow A_c^1$ as $\rho \rightarrow \infty$, i.e. this HYM connection is asymptotic to the canonical invariant connection, which recall is the pullback of a reducible HYM connection on a line bundle over $D= \mathbb{P}^1 \times \mathbb{P}^1$.
\end{remark}

\appendix

\section{Appendix}\label{Appendix}

This appendix will be used to study the function $h(\rho)$ and the conditions that ensure a given connection and Higgs field to extend over the zero section. 

\subsection{The function $h(\rho)$}

Studying the function $h(\rho)$ is a necessary step in order to use the results of chapter \ref{chapter2} in order to solve the ODE's in proposition \ref{lem:MainLem} to which the problem was reduced to at the end of section \ref{lsmall}. One starts with some preliminary explicit formulas. In terms of $r$
\begin{eqnarray}\label{fg}
\mathcal{F'} (r^2) = \left( \frac{3}{2} \right)^{\frac{1}{3}} \frac{\epsilon^{-\frac{2}{3}}}{\sqrt{ \frac{r^4 }{ \epsilon^4} -1}}  k^{\frac{1}{3}} \left(\frac{r^2}{\epsilon^2}\right) \ , \ \mathcal{G} (r) = \left( \frac{3 \epsilon^4}{2^4} \right)^{\frac{1}{3}}   k^{\frac{1}{3}} \left(\frac{r^2}{\epsilon^2}\right)  .
\end{eqnarray}
where $k: (1 , \infty) \rightarrow \mathbb{R}$ is the function defined by $k(x) = x \sqrt{x^2 - 1} - \log( \sqrt{x^2 - 1} + x )$. To write $\rho$ in terms of $r$ and using this function, insert \ref{fg} into equation \ref{s}, one has
\begin{eqnarray}
\rho(r) & = & \left( \frac{2}{3 \epsilon^4} \right)^{\frac{1}{3}} \int_{\epsilon}^{r} l k^{-\frac{1}{3}} \left(\frac{l^2}{\epsilon^2}\right) dl = \left( \frac{ \epsilon^2}{12} \right)^{\frac{1}{3}} \int_1^{\frac{r^2}{\epsilon^2}} k^{-\frac{1}{3}} \left( l \right) dl.
\end{eqnarray}
In order to see how the function $h^2(\rho)= \frac{1}{\epsilon^2} R_+ R_-  \mathcal{G}$ in terms $r$, it is useful to use $k$
\begin{eqnarray}\label{h2}
h^2(\rho(r)) & = & \left( \frac{3 \epsilon^4}{2^7} \right)^{\frac{1}{3}} \sqrt{ \frac{r^4}{ \epsilon^4} -1} k^{\frac{1}{3}} \left(\frac{r^2}{\epsilon^2}\right).
\end{eqnarray}

\begin{lemma}\label{hbehaviour}
The function $h(\rho)$ behaves for $\rho \ll 1$ as $h(\rho)= \rho + O(\rho^3)$ and for $\rho \gg 1$ one has $h(\rho)= O(\rho^{5/2})$.
\end{lemma}
\begin{proof}
Regarding the function $k: (1 , \infty) \rightarrow \mathbb{R}$, for $x$ close to $1$ one has the following expansions in terms of $\sqrt{x-1}$
\begin{eqnarray}\nonumber
k^{\frac{1}{3}}(x) = \frac{2^{\frac{5}{6}}}{3^{\frac{1}{3}}} \sqrt{x-1} + \frac{(x-1)^{3/2}}{10(2)^{\frac{1}{6}}(3)^{\frac{1}{3}}} \ , \ k^{-\frac{1}{3}}(x) = \frac{3^{\frac{1}{3}}}{2^{\frac{5}{6}}} \frac{1}{\sqrt{x-1}} - \frac{3^{\frac{1}{3}}}{20(2)^{\frac{5}{6}}} \sqrt{x-1}+...\nonumber
\end{eqnarray}
Inserting these expressions on $h^2$ and $\rho$, one has that for $\rho \ll 1$
\begin{eqnarray}\nonumber
\rho(r) & = & \frac{\epsilon^{\frac{2}{3}}}{\sqrt{2}} \left( \sqrt{ \frac{r^2}{\epsilon^2}-1 } - \frac{1}{60}\left( \frac{r^2}{\epsilon^2}  -1 \right)^{\frac{3}{2}} + ... \right) \\ \nonumber
h^2(r) & = & \frac{\epsilon^{\frac{4}{3}}}{2} \left(  \left( \frac{r^2}{\epsilon^2} -1 \right) + \frac{1}{20}\left( \frac{r^2}{\epsilon^2}  -1 \right)^2 + ... \right), \nonumber
\end{eqnarray}
hence, for small $\rho$, $h(\rho) \sim \rho + O(\rho^3)$. To get the behavior for large $\rho$, it is convenient to introduce one further coordinate given by $x=\cosh(t)$ for $t \in (0, \infty)$ since $x \in  (1 , \infty)$. Inverting this gives $t=\log(\sqrt{x^2 - 1} + x)$ and replacing it on $k$ shows that $h(\rho(t))  \sim  \epsilon^{2/3} e^{\frac{t}{2}} e^{\frac{t}{3}} = \epsilon^{2/3} e^{\frac{5t}{6}}$, while $\rho(t)  \sim  \epsilon^{2/3} \int e^t e^{-\frac{2t}{3}} = \epsilon^{2/3} e^{\frac{t}{3}}$ and the result follows.
\end{proof}

\subsection{Extending the Connection}

Studying the conditions that ensure a given connection and Higgs field to extend over the zero section is a necessary step for the proof of the main theorem \ref{Thm:TheoremStenzel}, which appears at the end of \ref{lsmall}. These conditions give rise to initial conditions at $\rho=0$ (the zero section) for the ODE's. These are the initial conditions that where stated in the hypothesis of proposition \ref{lem:MainLem}, which reduces the problem to that of solving the ODE's analyzed in the first part of chapter \ref{chapter2}.\\
It follows from formula \ref{RFmetric} for Stenzel's metric that the $1$-forms defined by
$$\omega_1 = \sqrt{2\frac{R_+ R_-}{r} \frac{d\mathcal{G}}{d r}} \theta_1 \ \ , \ \ \omega_{2,3} = \sqrt{\frac{R_+}{R_-} \mathcal{G}} \theta_{2,3} \ \ , \ \ \omega_{4,5} = \sqrt{\frac{R_-}{R_+} \mathcal{G}} \theta_{4,5},$$
have constant norm equal to $1$ and so are bounded. For a connection to extend it is a necessary condition that the curvature remains bounded.

\begin{lemma}\label{Alem:curvature}
Let $l=1$ and $A$ an invariant connection parametrized by the fields $A_i$. Let the $B_i$'s be the rescaled fields introduced in the statement of proposition \ref{newODE's}. Fix a gauge such that $B_2=0$ and suppose as well that $B_5=0$. Then, the curvature of the invariant connection can be written in this frame as
\begin{eqnarray}\nonumber
F_A & = &\left( I_4\omega^{23} + I_4 \omega^{45} + I_1 \omega^1 +I_8( \omega^{24} + \omega^{35} ) \right) \otimes T_1 \\ \nonumber
& & +I_2 \left( T_3 \otimes d\rho \wedge \omega^2 - T_2 \otimes d\rho \wedge \omega^3 \right) + I_3 \left( T_2 \otimes d\rho \wedge \omega^4 + T_3\otimes d\rho \wedge \omega^5 \right) \\ \nonumber
& &+I_6 \left( T_2 \otimes \omega^{12} + T_3 \otimes \omega^{13} \right) + I_7 ( T_2 \otimes \omega^{15} -T_3 \otimes \omega^{14}),
\end{eqnarray}
where
\begin{eqnarray}\nonumber
I_1 = \frac{1}{\epsilon^2 h^2(\rho)} \left( \frac{dB_1}{d\rho} - \frac{2 \dot{\mathcal{G}}}{r} B_1 \right) & , & I_8 = \frac{1}{\mathcal{G}} \left( \frac{B_1}{\mathcal{G}^2} - 2\frac{B_3 B_4}{R_+ R_-} \right)  \\ \nonumber
I_2 = \frac{1}{\epsilon h(\rho)} \left( \frac{dB_3}{d\rho} - \frac{\mathcal{G}}{rR_-^2} B_3 \right) & , & I_3 = \frac{1}{\epsilon h(\rho)} \left( \frac{dB_4}{d\rho} - \frac{\mathcal{G}}{rR_+^2} B_4 \right) \\ \nonumber
I_4 = \frac{1}{\epsilon^2 h^2(\rho)} \left( 4B_3^2-R_-^2 \right) & , & I_5= \frac{1}{\epsilon^2 h^2(\rho)} \left( 4B_4^2-R_+^2 \right) \\ \nonumber
I_6 = \left( \frac{B_4}{R_+} - 2\frac{B_1 B_3}{\mathcal{G}^2 R_-} \right) \frac{1}{R_+} \sqrt{ \frac{\mathcal{G}}{R_+ R_-} } & , & I_7 = \left( \frac{B_3}{R_-} - 2\frac{B_1 B_4}{\mathcal{G}^2 R_+} \right) \frac{1}{R_-} \sqrt{ \frac{\mathcal{G}}{R_+ R_-} } .
\end{eqnarray}
\end{lemma}
\begin{proof}
It follows from lemma \ref{lema1} that the curvature can be written as
\begin{eqnarray}\nonumber
F_A & = &\left( \left( 2 A_3^2 - \frac{1}{2}\right)\theta^{23} + \left(  2 A_4^2 - \frac{1}{2} \right) \theta^{45} +  \frac{d A_1}{d \rho} d\rho \wedge \theta^1 + ( A_1 -2 A_4 A_3 ) ( \theta^{24} + \theta^{35} ) \right) \otimes T_1 \\ \nonumber
& & + \frac{d A_3}{d \rho} \left( T_3 \otimes d\rho \wedge \theta^2 - T_2 \otimes d\rho \wedge \theta^3 \right) + \frac{d A_4}{d \rho} \left( T_2 \otimes d\rho \wedge \theta^4 + T_3\otimes d\rho \wedge \theta^5 \right) \\ \nonumber
& & +(A_4 - 2A_1 A_3) \left( T_2 \otimes \theta^{12} + T_3 \otimes \theta^{13} \right) + \left( A_3 - 2A_1A_4)\right) ( T_2 \otimes \theta^{15} -T_3 \otimes \theta^{14}).
\end{eqnarray}
Using the definition of the $B_i$'s in terms of the $A_i$'s, the definition of the bounded forms $\omega_i$ and the relations between $\rho, h, \mathcal{G}, R_+, R_-$ this turns into the formula in the statement.
\end{proof}

\begin{lemma}\label{lem:AextendingA}
The invariant connection $A$ from lemma \ref{Alem:curvature} extends over the zero section if and only if, for $\rho \ll 1$
$$B_1(\rho)= O(\rho^3) \ \ , \ \ B_3(\rho)= O(\rho^2) \ \ , \ \ B_4(\rho)= \frac{\epsilon}{2} + O(\rho^2).$$
\end{lemma}
\begin{proof}
The connection extends over the zero section if and only if the curvature does remain bounded. Since the forms $\omega_i$ are bounded, one concludes from lemma \ref{Alem:curvature} that this will be the case if and only if the $I_i$'s are bounded for $\rho \ll 1$. The fact that $I_5$ needs to stay bounded implies that
$$\left( 4B_4(\rho)^2-R_+(\rho)^2 \right)= O(h^2(\rho))= O(\rho^2).$$
Since $R_+^2 = \frac{\epsilon^2}{2} \left( \frac{r^2}{\epsilon^2} +1 \right)= \epsilon^2 + \frac{\epsilon^2}{2} \left( \frac{r^2}{\epsilon^2} -1 \right)= \epsilon^2 + O(\rho^2)$, then from the above one must have
$$B_4(\rho)=\frac{\epsilon^2}{4} + O(\rho^2),$$
and this gives the result in the statement. In the same way one can proceed to analyze $I_4$, which gives $4B_3^2-R_-^2 = O(\rho^2)$, but since $R_-^2=O(\rho^2)$, one concludes that $B_3^2=O(\rho^2)$ and so $B_3 = O(\rho)$. This is again the result in the statement and the only thing left to do is to compute the estimate on $B_1$. From $B_4(\rho)= \frac{\epsilon}{2} + O(\rho^2)$ and $B_3(\rho)=O(\rho)$. In fact inserting these into $I_8$ together with $\mathcal{G}=O(\rho)$ and $R_-=O(\rho)$, gives that
$$\rho^{-2} B_1 = O \left( \frac{B_3 B_4}{R_+ R_-} \right) = O(1),$$
from what it is straightforward to get $B_1(\rho)= O(\rho^2)$. So far, one has just used the boundedness of $I_4,I_5,I_8$ and obtained that
\begin{equation}\label{eq:AextendingA}
B_1(\rho)= O(\rho^2) \ \ , \ \ B_3(\rho)= O(\rho) \ \ , \ \ B_4(\rho)= \frac{\epsilon}{2} + O(\rho^2).
\end{equation}
One must analyze the behavior of the other $I_i$'s. Writing $B_1 = b_1 \rho^2$, $B_3=b_3 \rho$ and $B_4 = \frac{\epsilon}{2} + b_4 \rho^2$ one can see that the boundedness of $I_1,I_3,I_6$ are guaranteed just by the estimates in lemma \ref{eq:AextendingA}, while the boundedness of $I_7,I_8,I_2$ require respectively
$$b_3 = 2 \sqrt{2} \epsilon^{-\frac{7}{3}}b_1 b_4 \ \ , \ \ b_1 = 2 \sqrt{2} b_3 b_4 \ \ , \ \ b_3 =0.$$
Combining these implies that $b_1=b_3=0$ and the result follows.
\end{proof}

\begin{remark}
Moreover, a posteriori to proposition \ref{lem:MainLem}, bounded invariant connections satisfying the Calabi-Yau monopole equations, are known to satisfy a Bogomolny equation when restricted to the fibres of $T^* \mathbb{S}^3 \rightarrow \mathbb{S}^3$. Hence, by the main theorem of \cite{2Sibner1984} the condition that the curvature remains bounded is also a sufficient one for an invariant Calabi-Yau monopole to extend.
\end{remark}

\end{document}